\documentclass[12pt]{amsart}
\usepackage{a4wide,enumerate,xcolor}
\usepackage{amsmath,graphicx,comment,enumerate}
\usepackage{mathtools} 
\allowdisplaybreaks

\usepackage{enumitem}
\setlist[itemize]{label={$\bullet$}, leftmargin=12pt, itemsep=3pt}

\let\pa\partial
\let\na\nabla
\let\eps\varepsilon
\newcommand{\N}{{\mathbb N}}
\newcommand{\R}{{\mathbb R}}
\newcommand{\diver}{\operatorname{div}}

\newtheorem{theorem}{Theorem}
\newtheorem{lemma}[theorem]{Lemma}
\newtheorem{proposition}[theorem]{Proposition}

\newtheorem{definition}{Definition}

\begin{document}

\title[Chemotaxis compressible Navier--Stokes equations]{Chemotaxis compressible Navier--Stokes equations \\
with density-dependent viscosity \\
modeling vascular network formation}

\author[A. J\"ungel]{Ansgar J\"ungel}
\address{Institute of Analysis and Scientific Computing, TU Wien, Wiedner Hauptstra\ss e 8--10, 1040 Wien, Austria}
\email{juengel@tuwien.ac.at} 

\author[F. Philipp]{Flora Philipp}
\address{Institute of Analysis and Scientific Computing, TU Wien, Wiedner Hauptstra\ss e 8--10, 1040 Wien, Austria}
\email{flora.philipp@tuwien.ac.at} 

\date{\today}

\thanks{The authors acknowledge partial support from the Austrian Science Fund (FWF), grant 10.55776/F65, and from the Austrian Federal Ministry for Women, Science and Research and implemented by \"OAD, project MultHeFlo. This work has received funding from the European Research Council (ERC) under the European Union's Horizon 2020 research and innovation programme, ERC Advanced Grant NEUROMORPH, no.~101018153. For open-access purposes, the authors have applied a CC BY public copyright license to any author-accepted manuscript version arising from this submission.} 

\begin{abstract}
The existence of global weak solutions to the compressible Navier--Stokes equations for the density of endothelial cells and their velocity, coupled to a reaction--diffusion equation for the concentration of the chemoattractant, is established in a three-dimensional torus for energy-finite initial data. The coupling of the equations arises through the chemotaxis force, which contributes to the momentum balance equation, and the signal production due to the cells in the chemotaxis equation. The equations model the self-assembly of endothelial cells during the early stages of blood vessel formation. The existence result holds for adiabatic pressure exponents $\gamma>4/3$, matching the exponent found in the existence analysis for the degenerate Keller--Segel equations. The proof leverages an approximation via Korteweg and drag terms, the BD entropy inequality, and a construction of weak solutions that are renormalized in the velocity variable.
\end{abstract}

 

\maketitle


\section{Introduction}

In vasculogenesis, endothelial cells move through the extracellular matrix and interstitial fluid, guided by gradients of chemoattractants. The cells may self-assemble into a vascular network. Exhibiting fluid-type behavior and interface effects, the collective cell dynamics is described by mass transport, momentum balance, and chemotaxis forces. The mathematical model consists of the compressible Navier--Stokes equations for the cell density and velocity, coupled with a reaction--diffusion equations for the concentration of the chemoattractant. The existence of global weak solutions to this system for pressures with an adiabatic exponent $\gamma\ge 4$ was proved in \cite{AiAl16} and later extended to the range $\gamma>8/5$ in \cite{HuJu24}. The motivation of this paper is to investigate to what extent the admissible range of $\gamma$ can be extended by allowing for model variations. When density-dependent viscosities are assumed, we prove that we can choose $\gamma>4/3$. On the level of a Korteweg approximation, the full range $\gamma>1$ turns out to be admissible.

\subsection{Model equations}

The equations for the density $\rho(x,t)$ of the endothelial cells, their velocity $v(x,t)$, and the concentration $c(x,t)$ of the chemoattractant are given by
\begin{align}
  & \pa_t\rho + \diver(\rho v) = 0, \label{1.rho} \\
  &\pa_t(\rho v) + \diver(\rho v\otimes v) + \na p(\rho)
  = \nu\diver(\rho\mathbb{D}(v)) + \rho\na c - \frac{\rho v}{\zeta}, 
  \label{1.v} \\
  & \pa_t c = \Delta c - c + \rho\quad\mbox{in }\Omega,\ t>0,
  \label{1.c}
\end{align}
where $\mathbb{D}(v) = \frac12(\na v+\na v^T)$ is the symmetric part of the velocity gradient, $p(\rho) = \rho^\gamma$ is the pressure with the adiabatic exponent $\gamma>1$, $\nu>0$ is the viscosity constant, and $\zeta>0$ is the relaxation time. We assume that the domain equals the three-dimensional torus, $\Omega=\mathbb{T}^3$, and we impose the initial conditions
\begin{align}
  \sqrt{\rho}(0,\cdot) = \sqrt{\rho^0}, \quad 
  (\sqrt{\rho} v)(0,\cdot) = \sqrt{\rho^0}v^0, \quad
  c(0,\cdot) = c^0\quad\mbox{in }\Omega. \label{1.ic} 
\end{align}

Equations \eqref{1.rho}--\eqref{1.v} are the compressible Navier--Stokes equations with a density-depen\-dent viscosity and the chemotaxis force $\rho\na c$. Equation \eqref{1.c} is the chemical signal model including degradation of the signal and production due to the cells. Model \eqref{1.rho}--\eqref{1.c} was suggested in \cite[p.~1862]{ABS04} with constant viscosity. When the movement of the endothelial cells is modeled as a non-Newtonian fluid, the results of \cite{DMB23} suggest a degenerate viscosity that vanishes in vacuum to account for the nonstandard fluid behavior.

In the zero-inertia and zero-viscosity limit, equations \eqref{1.rho}--\eqref{1.c} reduce to the degenerate Keller--Segel equations $\pa_t\rho + \zeta\diver(\na p(\rho) - \rho\na c) = 0$ and $\pa_t c = \Delta c-c+\rho$. The existence of global classical solutions to this system was proved for functions $p(\rho)=\rho^\gamma$ with $\gamma>4/3$ \cite{TaWi12}, while there exist smooth solutions in case $\gamma\ge 4/3$ that blow up \cite{Win10}. This indicates that the range $\gamma>4/3$ is optimal to obtain a global existence result for \eqref{1.rho}--\eqref{1.ic}, and this result is our main contribution.

Without chemotaxis force, global weak solutions to the compressible Navier--Stokes equations with density-dependent viscosity were shown to exist for all $\gamma>1$ \cite{BrDe03,BrDe06}. This naturally raises the question of whether this range of $\gamma$ is possible in our model with chemotaxis force, allowing for a modification of the model. We demonstrate that this is achievable by introducing a Korteweg term of the form $\diver\mathbb{K}$, where the Korteweg tensor $\mathbb{K}$ is specifically given by $\mathbb{K} = (\kappa/2)\rho\na^2\log\rho$ with the capillarity constant $\kappa>0$. Then the Korteweg term becomes
\begin{align}\label{1.K}
  \diver\mathbb{K} = \kappa\rho\na\bigg(
  \frac{\Delta\sqrt\rho}{\sqrt\rho}\bigg).
\end{align} 
This expression can also be interpreted as a quantum force with the so-called Bohm potential $\Delta\sqrt\rho/\sqrt\rho$, used in quantum fluid modeling \cite{Jue10}. Although the presence of Korteweg terms may be motivated in the modeling of multiphase cellular fluids, we consider it as part of the approximation scheme for \eqref{1.rho}--\eqref{1.v}, introduced to establish the existence of weak solutions for $\kappa=0$. The existence of solutions to the compressible Navier--Stokes--Korteweg model with chemotaxis force for all $\gamma>1$ is a by-product of our analysis. 

\subsection{Mathematical tools}

We explain now the mathematical difficulties associated with the chemotaxis force. The idea of \cite{HuJu24}, where the viscosity was assumed to be constant, is to derive a priori estimates from the free energy
\begin{align}\label{1.E0}
  E_0(\rho,v,c) = \int_\Omega\bigg(\frac12\rho|v|^2 + h(\rho)
  + \frac12(|\na c|^2+c^2) - \rho c\bigg)dx,
\end{align}
where $h(\rho)=\rho^\gamma/(\gamma-1)$ is the internal energy density, $\frac12\rho|v|^2$ is the kinetic energy, and the remaining terms are associated to the chemotaxis energy. To prove that $E_0$ has a lower bound, we need an upper bound for $\rho c$. This was achieved by the inequality of Sugiyama \cite{Sug07}
\begin{align*}
  \int_\Omega\rho cdx \le \frac12\|h(\rho)\|_{L^1(\Omega)}
  + \frac14\|\na c\|_{L^2(\Omega)}^2
  + C_1(\gamma)\|c\|_{L^1(\Omega)}^{C_2(\gamma)},
\end{align*}
which requires the condition $\gamma>8/5$. The first two terms on the right-hand side can be absorbed by the energy and the $L^1(\Omega)$ of $c$ can be bounded in terms of the initial data $(\rho^0,c^0)$. We believe that the bound $\gamma>8/5$ cannot be easily improved. 

When the viscosity depends on the density as in \eqref{1.v}, Bresch and Desjardins discovered a new mathematical entropy that provides some additional regularity for the density. The so-called BD entropy
\begin{align}\label{1.BDent}
  E_{BD}(\rho,v) = \frac12\int_\Omega\rho|v+\nu\na\log\rho|^2 dx
\end{align}
yields an a priori bound for $\na\sqrt\rho$ in $L^2(\Omega)$ uniformly in time. It can be interpreted as the kinetic energy associated to the effective velocity, which is the sum of the fluid velocity $v$ and the osmotic velocity $\nu\na\log\rho$. By the Gagliardo--Nirenberg inequality (see Lemma \ref{lem.aux}), we can control the critical term $\rho c$ according to
\begin{align}\label{1.rhoc}
  \int_\Omega\rho cdx \le C(\delta) + \delta\int_\Omega(|\na\sqrt\rho|^2 
  + |\na c|^2)dx
\end{align}
for any $\delta>0$, and the right-hand side can be controlled by the BD entropy. 

Unfortunately, our situation is more involved. Indeed, we also need to estimate the term
\begin{align*}
  K_1 = \nu\int_0^t\int_\Omega\na\rho\cdot\na c dxd\tau
\end{align*}
arising from the chemotaxis force. We prove below (see the proof of \eqref{3.cdelta}) that
\begin{align*}
  K_1 \le C(c^0) + \int_0^t\int_\Omega\rho^2 dxd\tau,
\end{align*}
such that it is sufficient to estimate $\rho^2$. If $\gamma\ge 2$, the right-hand side can be estimated by the internal energy component of $E_0$. For $\gamma<2$, we need additional bounds derived from a combination of the free energy $E_0$ and the BD entropy $E_{BD}$. The associated dissipation contains, up to some factor, the dissipation
\begin{align*}
  D_1 = \frac{4}{\gamma}\nu\int_0^t\int_\Omega|\na\rho^{\gamma/2}|^2
  dxd\tau.
\end{align*}
Then, if $4/3<\gamma<2$, we use the inequality \begin{align}\label{1.rho2}
  \int_0^t\int_\Omega\rho^2 dxd\tau \le C(\delta)
  + \delta\int_0^t\int_\Omega|\na\rho^{\gamma/2}|^2 dxd\tau
\end{align}
for arbitrary $\delta>0$ (see Lemma \ref{lem.aux}), and the integral on the right-hand side can be absorbed by $D_1$. We already mentioned that the range $\gamma>4/3$ matches the bound for the degenerate Keller--Segel model \cite{TaWi12}, which can be seen as a limiting case of our model. 

To cover the full range $\gamma>1$, a stronger inequality is required:
\begin{align}\label{1.rho4}
  \int_0^t\int_\Omega\rho^2 dxd\tau \le C(\delta)
  + \delta\int_0^t\int_\Omega|\na\sqrt[4]{\rho}|^4 dxd\tau.
\end{align}
It is necessary to  regularize our equations to control the right-hand side of \eqref{1.rho4}. This is achieved by adding the Korteweg term \eqref{1.K}. Then a combination of $E_0$ and $E_{BD}$ yields the additional dissipation
\begin{align*}
  D_2 = \kappa\int_0^t\int_\Omega\rho|\na^2\log\rho|^2 dxd\tau,
\end{align*}
which can be estimated from below by
\begin{align}\label{1.qineq}
  D_2 \ge 2C_K\kappa\int_\Omega\big(|\Delta\sqrt\rho|^2 
  + |\na\sqrt[4]{\rho}|^4\big)dx \quad\mbox{with }C_K=\frac{1}{16},
\end{align}
proved in \cite[Appendix]{Jue10} and \cite[Lemma 2.1]{VaYu16}. This allows us to derive additional a priori estimates. In fact, for sufficiently large $M>0$, we show below that
\begin{align*}
  \frac{d}{dt}(E_0 + E_1 + ME_{BD}) + D \le C(\kappa)
  \quad\mbox{for }0<t<T,
\end{align*}
where $E_1=\kappa\int_\Omega|\na\sqrt\rho|^2dx$ is the Korteweg energy, $D$ is the dissipation containing gradient bounds, and the constant $C(\kappa)>0$ depends on $\kappa$ and the initial data. If $\gamma>4/3$, the constant can be chosen independently of $\kappa$, allowing for the limit $\kappa\to 0$. 

We detail now our approximation scheme. It is well known that the solution to the compressible Navier--Stokes equations requires a careful regularization strategy. We consider two regularization levels to solve \eqref{1.rho}--\eqref{1.c}. First, we add the Korteweg term \eqref{1.K} and some drag terms with parameters $(r_0,r_1)$, leading to the following chemotaxis compressible Navier--Stokes--Korteweg system:
\begin{align}
  & \pa_t\rho + \diver(\rho v) = 0, \quad
  \pa_t c = \Delta c - c + \rho\quad\mbox{in }\Omega,\ t>0, 
  \label{1a.rhoc} \\
  &\pa_t(\rho v) + \diver(\rho v\otimes v) + \na p(\rho)
  = \nu\diver(\rho\mathbb{D}(v)) + \rho\na c - \frac{\rho v}{\zeta}
  \label{1a.v} \\
  &\phantom{xx}+ \kappa\na\bigg(\frac{\Delta\sqrt\rho}{\sqrt\rho}\bigg)
  - r_0v - r_1\rho|v|^2v. \nonumber
\end{align}
Second, we add the diffusion $\eps\Delta\rho$ and some higher-order regularization with parameters $(\delta,\eta,\mu)$ as in \cite[Sec.~2]{VaYu16}; see Section \ref{sec.approx} below. 

The strategy of the proof is as follows. We first prove the existence of solutions to the regularized Navier--Stokes--Korteweg system with parameters $(\kappa,r_0,r_1;\delta,\eps,\eta,\mu)$ by means of a Faedo--Galerkin approximation with dimension $n\in\N$. A modified energy inequality yields estimates uniform in $n$ such that we can pass to the limit $n\to\infty$. Then we prove the BD entropy inequality, which provides additional uniform bounds allowing us to perform the limit $(\delta,\eps,\eta,\mu)\to 0$. This gives a weak solution to the compressible Navier--Stokes--Korteweg system \eqref{1a.rhoc}--\eqref{1a.v} for all $\gamma>1$. Following \cite{LaVa18}, we show that this solution is in fact a renormalized solution (and vice versa). The final step is the limit $(\kappa,r_0,r_1)\to 0$ leading to a weak solution to system \eqref{1.rho}--\eqref{1.ic} for $\gamma>4/3$, recalling that the restriction of the range of $\gamma$ originates from inequality \eqref{1.rho2}.

\subsection{State of the art}

The first result of the existence of global weak solutions to the three-dimensional compressible Navier--Stokes equations with constant viscosity and nonnegative initial data was proved by Lions \cite[Theorem 7.2]{Lio98} for $\gamma>9/5$. This result was extended by Feireisl and co-workers \cite{FNP01} to the range $\gamma>3/2$. The full range $\gamma>1$ seems to be admissible only under additional assumptions like symmetric initial data \cite{JiZh03}. These results rely on the condition that the viscosity coefficient is bounded from below by a positive constant to achieve uniform bounds on the gradient of the velocity.

When density-dependent viscosities are considered, the velocity cannot be defined when the density vanishes. This difficulty was first overcome by Bresch et al.\ \cite{BDL03} by proving the stability of a variant of the compressible Navier--Stokes--Korteweg equations, later improved by Bresch and Desjardins \cite{BrDe03} by removing the Korteweg term. This was possible thanks to a new mathematical entropy inequality, derived from the BD entropy \eqref{1.BDent}. With this information, the existence of global weak solutions in the presence of appropriate drag terms or singular pressure close to vacuum was obtained. These additional terms could be removed by Mellet and Vasseur \cite{MeVa07} thanks to a new Mellet--Vasseur inequality, giving a uniform estimate for $\int_\Omega\rho|v|^2\log(1+|v|^2)dx$ for smooth solutions. An approximation scheme was provided in \cite{LiXi15}, and the Mellet--Vasseur inequality for weak solutions was proved in \cite{VaYu16inv}. When the Korteweg term \eqref{1.K} is present, such an estimate cannot be expected, except for $\kappa<\nu$ \cite[Theorem 2.2]{AnSp18}. 

Another approach was introduced by Lacroix-Violet and Vasseur \cite{LaVa18}. The authors prove an existence result for the compressible Navier--Stokes equations with Korteweg term \eqref{1.K} and then perform the limit $\kappa\to 0$. The method is based on the construction of weak solutions that are renormalized in the velocity variable, first used in \cite{VaYu16inv} for the case $\kappa=0$. More general nonlinear density-dependent viscosities have been studied too, see, e.g., the work \cite{BVY22} without Korteweg term and \cite{ABS25} with Korteweg term.

As already mentioned, we will follow the approach of \cite{LaVa18} to prove the existence of weak solutions to \eqref{1.rho}--\eqref{1.ic}. The novelty of this paper lies in the treatment of the chemotaxis term, which poses significant challenges for small values of $\gamma>1$ due to its quadratic structure.


\subsection{Definitions}

We give the definitions of weak and renormalized solutions. To this end, we introduce the following tensors:
\begin{align}
  \sqrt{\kappa\rho}\mathbb{S}_\kappa
  &= \kappa\sqrt{\rho}\big(\na^2\sqrt{\rho} - 4\na\sqrt[4]{\rho}
  \otimes\na\sqrt[4]{\rho}\big), \label{1.Skappa} \\
  \sqrt{\nu\rho}\mathbb{T}_\nu &= \nu\na(\rho v) 
  - 2\nu\sqrt{\rho} v\otimes\na\sqrt{\rho}. \label{1.Tnu}
\end{align}
The Korteweg tensor $\sqrt{\kappa\rho}\mathbb{S}_\kappa$ satisfies
\begin{align*}
  \diver(\sqrt{\kappa\rho}\mathbb{S}_\kappa)
  = \kappa\rho\na\frac{\Delta\sqrt\rho}{\sqrt\rho}
\end{align*}
and accounts for the Korteweg regularization, while the tensor $\mathbb{T}_\nu$ equals $\sqrt{\nu\rho}\na v$, both for positive smooth functions $\rho$. The tensors are needed, since expressions like $1/\sqrt\rho$ and $\na v$ may be not defined. We also introduce the symmetric part of $\mathbb{T}_\nu$, $\mathbb{S}_\nu = \frac12(\mathbb{T}_\nu+\mathbb{T}_\nu^T)$ and we set $\Omega_T:=(0,T)\times\Omega$. 

Our a priori estimates, stated below, yield the following regularity for $\kappa\ge 0$:
\begin{equation}\label{1.regul}
\begin{aligned}
  & \rho\in L^\infty(0,T;L^\gamma(\Omega)), \quad
  \rho^{\gamma/2}\in L^2(0,T;H^1(\Omega)), \quad 
  \na\sqrt\rho,\, \sqrt{\rho}v\in L^\infty(0,T;L^2(\Omega;\R^3))\\
  & \mathbb{T}_\nu\in L^2(\Omega_T;\R^{3\times 3}), \quad 
  c\in L^\infty(0,T;H^1(\Omega))\cap H^1(0,T;L^2(\Omega)), \\
  & \sqrt{r_0}v\in L^2(\Omega_T),\quad
  \sqrt\kappa\Delta\sqrt\rho\in L^2(\Omega_T), \quad
  \na\sqrt[4]{\kappa\rho},\,\sqrt[4]{r_1\rho}v\in L^4(\Omega_T;\R^3).
\end{aligned}
\end{equation}

\begin{definition}[Weak solution to the chemotaxis Navier--Stokes--Korteweg system]\label{def.weak}
The \\ tri\-ple $(\sqrt\rho,\sqrt\rho v,c)$ is called a {\em weak solution} to \eqref{1a.rhoc}--\eqref{1a.v} on $[0,T]$ with initial conditions \eqref{1.ic} if 
\begin{align}
  & \int_0^T\int_\Omega\big(\rho\pa_t\phi + \rho v\cdot\na\phi
  \big)dxdt = 0, \label{1.wrho} \\
  & \int_0^T\int_\Omega\big(\rho v\cdot\pa_t\psi 
  + (\rho v\otimes v):\na\psi - \na p(\rho)\cdot\psi
  + \rho\na c\cdot\psi\big)dxdt
  \label{1.wv} \\
  &\phantom{xx}
  = \int_0^T\int_\Omega\bigg(\frac{\rho v}{\zeta}\cdot\psi
  + \big(\sqrt{\nu\rho}\mathbb{S}_\nu
  + \sqrt{\kappa\rho}\mathbb{S}_\kappa\big):\na\psi + r_0\rho v\cdot\psi
  + r_1\rho|v|^2 v\cdot\psi\bigg)dxdt, \nonumber \\
  & \int_0^T\int_\Omega\big(c\pa_t\phi - \na c\cdot\na\phi
  - c\phi + \rho\phi\big)dxdt = 0. \nonumber 
\end{align}
for all $\phi\in C_0^\infty(\Omega_T)$, $\psi\in C_0^\infty(\Omega_T;\R^3)$, $\mathbb{S}_\kappa$ is defined in \eqref{1.Skappa}, $\mathbb{S}_\nu=\frac12(\mathbb{T}_\nu+\mathbb{T}_\nu^T)$, $\mathbb{T}_\nu$ is given by \eqref{1.Tnu}, the regularity \eqref{1.regul} is satisfied, and the initial data \eqref{1.ic} is fulfilled in the sense of distributions.
\end{definition}

\begin{definition}[Weak solution to the chemotaxis Navier--Stokes 
system]
The triple $(\sqrt\rho,$ $\sqrt\rho v,c)$ is called a {\em weak solution} to \eqref{1.rho}--\eqref{1.c} on $[0,T]$ with initial conditions \eqref{1.ic} if it is a weak solution to \eqref{1.ic}, \eqref{1a.rhoc}--\eqref{1a.v} with $(\kappa,r_0,r_1)=0$.
\end{definition}

\begin{definition}[Renormalized solution to the chemotaxis Navier--Stokes--Korteweg system]\label{def.renorm}
\mbox{\,}\linebreak
The triple $(\sqrt\rho,\sqrt\rho v,c)$ is called a {\em weak solution} to \eqref{1a.rhoc}--\eqref{1a.v} on $(0,T)$ with initial conditions \eqref{1.ic} if for any $\varphi\in W^{3,\infty}(\R^3)$, there exists $C_1>0$ such that 
\begin{align}\label{1.varphi}
  |y_j\varphi(y)| + \bigg|y_j\frac{\pa\varphi}{\pa y_k}(y)\bigg|
  \le C_1\quad\mbox{for all }y\in\R^3,\ j,k=1,2,3;
\end{align}
there exist $C_2>0$, depending on the $L^2(\Omega_T)$ norm of $\mathbb{S}_\kappa$ and $\mathbb{T}_\nu$, and measures $R_\varphi$, $Q_\varphi^{ijk} \in \mathcal{M}(\Omega_T)$ such that
\begin{align*}
  \|R_\varphi\|_{\mathcal{M}(\Omega_T)}
  + \sum_{i,j,k=1}^3\|Q_\varphi^{ijk}\|_{\mathcal{M}(\Omega_T)} 
  \le C_2\|\varphi''\|_{L^\infty(\R^3)};
\end{align*}
the equations are satisfied in the sense
\begin{align}
  & \int_0^T\int_\Omega\big(\rho\pa_t\phi + \rho v\cdot\na\phi
  \big)dxdt = 0, \nonumber \\ 
  &\langle R_\varphi,\phi\rangle
  = \int_0^T\int_\Omega(\rho\varphi(v)\pa_t\phi
  + \rho\varphi(v)v\cdot\na\phi - \na p(\rho)\cdot\varphi'(v)\phi 
  + \rho\na c\cdot\varphi'(v)\phi)dxdt \nonumber \\ 
  &\phantom{xxxxxxx} - \int_0^T\int_\Omega\bigg(
  \frac{\rho v}{\zeta}\cdot\varphi'(v)\phi
  + \big(\sqrt{\nu\rho}\mathbb{S}_\nu 
  + \sqrt{\kappa\rho}\mathbb{S}_\kappa\big):(\varphi'(v)\otimes\na\phi)
   \bigg)dxdt \nonumber \\
  &\phantom{xxxxxxx}- \int_0^T\int_\Omega\big(
  r_0v\cdot\varphi'(v)\phi + r_1\rho|v|^2 v\cdot\varphi'(v)
  \phi\big)dxdt, \nonumber \\
  &\int_0^T\int_\Omega\big(c\pa_t\phi + c\Delta\phi
  - c\phi + \rho\phi\big)dxdt = 0 \nonumber 
\end{align}
for all $\phi\in C_0^\infty(\Omega_T)$; $\mathbb{S}_\kappa$ is defined in \eqref{1.Skappa}, $\mathbb{S}_\nu = \frac12(\mathbb{T}_\nu+\mathbb{T}_\nu^T)$, $\mathbb{T}_\nu$ is given by 
\begin{align}\label{1.T}
  \int_0^T\int_\Omega\sqrt{\nu\rho}\frac{\pa\varphi}{\pa v_i}(v)
  (\mathbb{T}_\nu)_{jk}\psi dxdt
  &= -\nu\int_0^T\int_\Omega\bigg(\rho\frac{\pa\psi}{\pa x_j}
  + 2\sqrt{\rho}\frac{\pa\sqrt\rho}{\pa x_j}\bigg)
  \frac{\pa\varphi}{\pa v_i}(v)v_k dxdt \\
  &\phantom{xx}+ \langle Q_\varphi^{ijk},\psi\rangle \nonumber 
\end{align}
for $i,j,k=1,2,3$ and $\psi\in C_0^\infty(\Omega_T)$; the regularity \eqref{1.regul} is satisfied; and the initial data is fulfilled in the sense of distributions.
\end{definition}


\subsection{Main results}

Let $\nu>0$ and $\zeta>0$. We impose the following conditions on the initial conditions:
\begin{align}\label{1.init}
  \rho^0\in L^\gamma(\Omega), \quad
  \sqrt{\rho^0},\ c^0\in H^1(\Omega), \quad
  \sqrt{\rho^0}v^0\in L^2(\Omega;\R^3), \quad
  r_0\log\rho^0\in L^1(\Omega). 
\end{align}
We first show the existence of global weak solutions to the chemotaxis compressible Navier--Stokes--Korteweg system.

\begin{theorem}[Existence for the chemotaxis Navier--Stokes--Korteweg system]\label{thm.nske}
Let $\kappa$, $r_0$, \\ $r_1>0$, $T>0$, $\gamma>1$, and let \eqref{1.init} hold. Then there exists a weak solution to \eqref{1.ic}, \eqref{1a.rhoc}--\eqref{1a.v} on $[0,T]$. If $\gamma > 4/3$, the regularity results \eqref{1.regul} are independent of $(\kappa,r_0,r_1)$.
\end{theorem}

If $\gamma > 4/3$, we can pass to the limit $(\kappa,r_0,r_1)\to 0$ and obtain a weak solution to the chemotaxis compressible Navier--Stokes system.

\begin{theorem}[Existence for the chemotaxis Navier--Stokes system]\label{thm.nse}
Let $T>0$, $\gamma>4/3$, and let \eqref{1.init} hold. Then there exists a weak solution to \eqref{1.rho}--\eqref{1.ic} on $[0,T]$. 
\end{theorem}

The paper is organized as follows. Inequalities \eqref{1.rhoc}--\eqref{1.rho4} are shown in Section \ref{sec.aux}. Theorems \ref{thm.nske} and \ref{thm.nse} are proved in Sections \ref{sec.nske} and \ref{sec.nse}, respectively.


\section{An auxiliary lemma}\label{sec.aux}

We show the following result which is needed to estimate the terms arising from the chemotaxis interaction.

\begin{lemma}\label{lem.aux}
For arbitrary $\sigma_1$, $\sigma_2>0$, there exists $C>0$, depending on $\sigma_i$ and the $L^1(\Omega)$ norms of $\rho$ and $c$, such that for all smooth functions $\rho$, $c$, defined on $\Omega$, 
\begin{align}\label{2.rhoc}
  \int_\Omega\rho c dx \le 
  \begin{dcases}
  C + \sigma_1\int_\Omega\rho^\gamma dx 
  + \sigma_2\int_\Omega|\na c|^2 dx 
  &\mbox{if }\gamma>8/5, \\
  C + \sigma_1\int_\Omega|\na\sqrt{\rho}|^2 dx
  + \sigma_2\int_\Omega|\na c|^2 dx.
  \end{dcases}
\end{align}
Furthermore, for any $\sigma>0$, there exists $C>0$, depending on $\sigma$ and the $L^1(\Omega)$ norm of $\rho$, such that for all smooth functions $\rho$, defined on $(0,t)\times\Omega$,
\begin{align}\label{2.rho2}
  \int_0^t\int_\Omega\rho^2 dxd\tau \le 
  \begin{dcases}
  Ct + \int_0^t\int_\Omega\rho^\gamma dxd\tau
  &\mbox{if }\gamma\ge 2, \\
  Ct + \sigma\int_0^t\int_\Omega|\na\rho^{\gamma/2}|^2 dxd\tau
  &\mbox{if }4/3<\gamma<2, \\
  Ct + \sigma\int_0^t\int_\Omega|\na\sqrt[4]{\rho}|^4 dxd\tau. 
  \end{dcases}
\end{align}
\end{lemma}

\begin{proof}
The proof of \eqref{2.rhoc} relies on the inequality
\begin{align*}
  \int_\Omega\rho cdx \le \sigma_1\|\rho\|_{L^m(\Omega)}^m
  + \sigma_2\|\na c\|_{L^2(\Omega)}^2 + C_1(\sigma_1,\sigma_2)
  \|c\|_{L^1(\Omega)}^{C_2(m)}
\end{align*}
for functions $\rho\in L^m(\Omega)$, $c\in H^1(\Omega)$, where $\sigma_1$, $\sigma_2>0$ are arbitrary, $C_1(\sigma_1,\sigma_2)$, $C_2(m)>0$, and $m>8/5$ \cite[Appendix B]{Sug07}. We need to control the first term on the right-hand side. If $\gamma>8/5$, we choose $m=\gamma$. In the general case, we choose $m=5/3$ and apply the Gagliardo--Nirenberg inequality with $\theta=3/5$:
\begin{align*}
  \|\rho\|_{L^m(\Omega)}^m &= \|\sqrt\rho\|_{L^{2m}(\Omega)}^{2m}
  \le C\|\na\sqrt\rho\|_{L^2(\Omega)}^{2m\theta}
  \|\sqrt{\rho}\|_{L^2(\Omega)}^{2m(1-\theta)}
  + C\|\sqrt\rho\|_{L^2(\Omega)}^{2m}  \\
  &\le C + C\|\na\sqrt\rho\|_{L^2(\Omega)}^{2},
\end{align*}
where we used $m\theta =1$, and $C>0$ depends on the $L^1(\Omega)$ norm of $\rho$. This shows \eqref{2.rhoc}.

We turn to the proof of \eqref{2.rho2}. The inequality is clear if $\gamma=2$. If $\gamma>2$, there exists $C>0$ such that $\rho^2\le C + \rho^\gamma$, which proves the second case. Let $4/3<\gamma<2$. Then, by the Gagliardo--Nirenberg inequality with $\theta= 3\gamma/(6\gamma-2)<1$ (since $\gamma>1$),
\begin{align*}
  \int_\Omega\rho^2 dx 
  &= \|\rho^{\gamma/2}\|_{L^{4/\gamma}(\Omega)}^{4/\gamma}
  \le C\|\na\rho^{\gamma/2}\|_{L^{2}(\Omega)}^{4\theta/\gamma}
  \|\rho^{\gamma/2}\|_{L^{2/\gamma}(\Omega)}^{4(1-\theta)/\gamma}
  + C \|\rho^{\gamma/2}\|_{L^{2/\gamma}(\Omega)}^{4/\gamma} \\
  &\le C+C\|\na\rho^{\gamma/2}\|_{L^{2}(\Omega)}^{4\theta/\gamma}
  \le C(\sigma) + \sigma\|\na\rho^{\gamma/2}\|_{L^{2}(\Omega)}^2,
\end{align*} 
where the last step follows from Young's inequality using $4\theta/\gamma<2$ (at this point, we need $\gamma>4/3$). Finally, for any $\gamma>1$, applying the Gagliardo--Nirenberg inequality again (with $\theta=3/8$), 
\begin{align*}
  \int_\Omega\rho^2 dx &= \|\sqrt[4]{\rho}\|_{L^8(\Omega)}^8
  \le C\|\na\sqrt[4]{\rho}\|_{L^{4}(\Omega)}^{8\theta}
  \|\sqrt[4]{\rho}\|_{L^4(\Omega)}^{8(1-\theta)}
  + C\|\sqrt[4]{\rho}\|_{L^4(\Omega)}^8 \\
  &\le C(\sigma) + \sigma\|\na\sqrt[4]{\rho}\|_{L^{4}(\Omega)}^4
\end{align*}
for an arbitrary $\sigma>0$, which is possible since $8\theta=3<4$.
\end{proof}


\section{Proof of Theorem \ref{thm.nske}}\label{sec.nske}

We formulate an approximate system by adding, similarly as in \cite[Sec.~2]{VaYu16}, higher-order terms in the density and velocity. This system is approximated by a Faedo--Galerkin scheme.  Estimates uniform in the Faedo--Galerkin dimension follow from an approximate energy inequality. To pass to the limit in the higher-order terms, we formulate a BD entropy inequality, which yields additional estimates uniform in the regularization parameters. The de-regularization limit is then performed by following the compactness arguments of \cite{VaYu16}. 

\subsection{Solution of an approximate system}\label{sec.approx}

We prove the existence of a weak solution to the chemotaxis compressible Navier--Stokes--Korteweg equations \eqref{1a.rhoc}--\eqref{1a.v}. We solve first the approximate system
\begin{align}
  & \pa_t\rho + \diver(\rho v) = \eps\Delta\rho, 
  \quad \pa_t c = \Delta c - c + \rho\quad\mbox{in }\Omega,\ t>0, 
  \label{3.rhoc} \\
  & \pa_t(\rho v) + \diver(\rho v\otimes v) + \na p(\rho)
  = \nu\diver(\rho\mathbb{D}(v)) + \rho\na c - \frac{\rho v}{\zeta}
   + Q_K + Q_{ho} \label{3.v}
\end{align}
where the Korteweg/drag regularization $Q_K=Q_K(\kappa,r_0,r_1)$ and the higher-order regularization $Q_{ho}=Q_{ho}(\delta,\eps,\eta,\mu)$ are defined by, respectively,
\begin{align*}
  Q_K(\kappa,r_0,r_1) &= \kappa\rho\na\bigg(
  \frac{\Delta\sqrt\rho}{\sqrt\rho}\bigg) - r_0v
  - r_1\rho|v|^2 v, \\
  Q_{ho}(\delta,\eps,\eta,\mu) &= \eps\diver(v\otimes\na\rho)
  - \mu\Delta^2v + \eta\na\rho^{-3} + \delta\rho\na\Delta^5\rho,
\end{align*}
and we impose smooth initial conditions \eqref{1.ic} such that $\rho^0$ is strictly positive. The regularization parameters $(\eps,r_0,r_1,\mu,\eta,\delta)$ are positive; we suppose additionally that $\eps\le 1$.

The higher-order regularization is similar to \cite[(2.6)]{VaYu16}. The diffusion term $\eps\Delta\rho$ in the mass equation implies strict lower and upper bounds for the density via the maximum principle, while the term $\eps\diver(v\otimes\na\rho)$ is added to preserve the energy structure. The higher-order term $\mu\Delta^2 v$ yields more regularity for the velocity, which is needed to derive the BD entropy inequality. The remaining expressions $\eta\na\rho^{-3}$ and $\delta\rho\na\Delta^5\rho$ yield the strict positivity of the density by means of the inequality 
\begin{align}\label{3.pos}
  \|\rho^{-1}\|_{L^\infty(\Omega)}
  \le C\big(1+\|\rho^{-1}\|_{L^3(\Omega)}\big)^3
  \big(1+\|\rho\|_{H^k(\Omega)}\big)^2,
\end{align}
which holds for all $\rho\in H^k(\Omega)$ with $k>7/2$ such that $\rho^{-1}\in L^3(\Omega)$ \cite[(13)]{BrDe06}. This allows us to use later the test function $\na\log\rho$ to derive the BD entropy inequality.

System \eqref{3.rhoc}--\eqref{3.v} is solved by a Faedo--Galerkin approximation as in \cite[Chap.~7]{Fei04} or \cite[Sec.~7.7]{NoSt04}. The idea is to choose a basis $(e_k)$ of eigenfunctions of $-\Delta$ with domain $\Omega$ and the finite-dimensional space $X_n=\operatorname{span}\{e_1,\ldots,e_n\}^3\subset L^2(\Omega;\R^3)$. A fixed-point argument shows the existence of a time $T_n>0$ and a solution $(\rho_n,c_n,v_n)\in C^1([0,T_n];C^\infty(\Omega))^2\times C^1([0,T_n];X_n)$ to \eqref{3.rhoc} and
\begin{align}\label{3.vn}
  \int_\Omega\big(&\rho_n v_n\cdot\theta(T_n)
  - \rho^0v^0\cdot\theta(0)\big)dx 
  = \int_0^{T_n}\int_\Omega\bigg(\rho_nv_n\cdot\pa_\tau\theta
  + \rho_n(v_n\otimes v_n):\na\theta \\
  &+ p(\rho_n)\diver\theta 
  - \nu v_n\cdot\diver(\rho_n\mathbb{D}(\theta))
  + \rho_n\na c_n\cdot\theta - \frac{\rho_n v_n}{\zeta}\cdot\theta
  \bigg)dx d\tau \nonumber \\
  &+ \int_0^{T_n}\int_\Omega\bigg(\kappa\rho_n\na\bigg(
  \frac{\Delta\sqrt{\rho_n}}{\sqrt{\rho_n}}\bigg)\cdot\theta
  - r_0v_n\cdot\theta - r_1\rho_n|v_n|^2v_n\cdot\theta\bigg)dxd\tau 
  \nonumber \\
  &- \int_0^{T_n}\int_\Omega\big(\eps(v_n\otimes\na\rho_n):\na\theta
  + \mu v_n\cdot\Delta^2\theta + \eta\rho_n^{-3}\diver\theta
  - \delta\rho_n\na\Delta^5\rho_n\cdot\theta\big)dxd\tau \nonumber 
\end{align}
for all $\theta\in C^1([0,T_n];X_n)$. The initial conditions are $\rho_n(0,\cdot)=\rho^0$, $v_n(0,\cdot)=P_nv^0$, $c_n(0,\cdot)=c^0$ in $\Omega$, where $P_n$ the projection on $X_n$. Let the initial data satisfy $0<\rho_-^0\le\rho_n^0(x)\le\rho_+^0$, $0<c_-^0\le c_n^0(x)\le c_+^0$ for a.e.\ $x\in\Omega$. The maximum principle implies the bounds
\begin{align*}
  \rho_-(t)\le\rho_n(t,x)\le\rho_+(t), \quad
  \min\{c_-^0,\rho_-^0\}\le c_n(t,x)\le \max\{c_+^0,\rho_+^0\}
\end{align*}
for $(t,x)\in\Omega_{T_n}$, where
\begin{align*}
  \rho_\pm(t) :=\rho_\pm^0\exp\bigg(\pm\int_0^{T_n}
  \|\diver v_n\|_{L^\infty(\Omega)}d\tau\bigg) > 0.
\end{align*}

To prove that the local solution $(\rho_n,c_n,v_n)$ exists on the whole time interval $[0,T]$, we need a uniform estimate for $v_n(t)$ in $X_n$ on $[0,T_n]$. This is achieved by the following approximate energy inequality. Recall that $h(\rho) = \rho^\gamma/(\gamma-1)$ and let the modified energy be given by 
\begin{align}\label{3.EK}
  E_{K}(\rho,v,c) = \int_\Omega\bigg(\frac{1}{2}\rho|v|^2
  + h(\rho) + \kappa|\na\sqrt\rho|^2 + \frac12(|\na c|^2+c^2)\bigg)dx,
\end{align}
and introduce the higher-order energy
\begin{align}\label{3.Eho}
  E_{ho}(\rho) = \int_\Omega\bigg(\frac{\eta}{4}\rho^{-3}
  + \frac{\delta}{2}|\na\Delta^2\rho|^2\bigg)dx.
\end{align}
Notice that we have not added the chemotaxis part $-\rho c$ as in \eqref{1.E0} such that this expression will appear on the right-hand side of the energy inequality.

\begin{lemma}[Approximate energy inequality]
Let $(\rho_n,c_n,v_n)$ be the local solution to \eqref{3.rhoc} and \eqref{3.vn}. Then, for $0<t\le T_n$,
\begin{align}\label{3.aei}
  E_K&((\rho_n,v_n,c_n)(t)) + E_{ho}(\rho(t))
  + \int_0^t\big(D_K(\rho_n,v_n,c_n) + D_{ho}(\rho_n,v_n,c_n)
  \big)d\tau \\
  &\le (1+\eps)E_K(\rho^0,v^0,c^0) + E_{ho}(\rho^0)
  + \int_\Omega(\rho_nc_n)(t)dx 
  + \eps\int_0^t\int_\Omega\rho_n^2 dxd\tau. \nonumber 
\end{align}
where the dissipation terms are given by
\begin{align}
  D_K(\rho_nv_n,c_n) &= \int_\Omega\bigg(\frac{\rho_n|v_n|^2}{\zeta}
  + r_0|v_n|^2 + r_1\rho_n|v_n|^4 + \nu\rho_n|\mathbb{D}v_n|^2
  + |\pa_\tau c_n|^2\bigg)dx, \nonumber \\ 
  D_{ho}(\rho_n,v_n,c_n) &= \int_\Omega\bigg(\frac{4\eps}{\gamma}
  |\na\rho_n^{\gamma/2}|^2 + C_K\eps\kappa  
  \big(|\na\sqrt[4]{\rho_n}|^4 + |\Delta\sqrt{\rho_n}|^2\big) 
  \nonumber \\
  &\phantom{xx}+ \mu|\Delta v_n|^2 
  + \frac{4\eps \eta}{3}|\nabla\rho^{-3/2}|^2
  + \eps\delta|\Delta^3\rho_n|^2\bigg)dx, \nonumber 
\end{align}
and the constant $C_K=1/16$ comes from inequality \eqref{1.qineq}.
\end{lemma}

\begin{proof}
The proof is similar to that one of \cite[Lemma 3]{HuJu24}. However, this proof only holds for $\gamma>8/5$. Therefore, we need to estimate partially in a different way. We use the test function $v_n$ in the momentum equation \eqref{3.vn}, the test function $\frac12|v_n|^2 + h'(\rho_n) - \kappa\Delta\sqrt{\rho_n}/\sqrt{\rho_n} - \delta\Delta^5\rho_n - \frac34\eta\rho_n^{-4}$ in the weak formulation of the mass equation in \eqref{3.rhoc}, and sum both expressions. A computation yields
\begin{align}\label{3.aux0}
  \int_\Omega&\bigg(\frac12\rho_n|v_n|^2 + h(\rho_n)
  + \kappa|\na\sqrt{\rho_n}|^2 + \frac{\eta}{4}\rho_n^{-3}
  + \frac{\delta}{2}|\na\Delta^2\rho_n|^2\bigg)dx
  \bigg|_{\tau=0}^{\tau=t} \\
  &\phantom{xx}+ \int_0^t\int_\Omega\bigg(\frac{\rho_nv_n}{\zeta}
  + r_0|v_n|^2 + r_1\rho_n|v_n|^4 + \nu\rho_n|\mathbb{D}v_n|^2
  \bigg)dxd\tau \nonumber \\
  &\phantom{xx}+ \int_0^t\int_\Omega\bigg(
  \frac{4\eps}{\gamma}|\na\rho_n^{\gamma/2}|^2
  + \mu|\Delta v_n|^2 + \frac{4\eps \eta}{3}|\nabla\rho^{-3/2}|^2
  + \eps\delta|\Delta^3\rho_n|^2\bigg)dxd\tau \nonumber \\
  &= \eps\kappa\int_0^t\int_\Omega\na\rho_n\cdot\na
  \bigg(\frac{\Delta\sqrt{\rho_n}}{\sqrt\rho_n}\bigg)dxd\tau
  + \int_0^t\int_\Omega\rho_n\na c_n\cdot v_n dxd\tau. \nonumber 
\end{align}
It remains to control the terms on the right-hand side. The identity $2\rho_n\na(\Delta\sqrt{\rho_n}/\sqrt{\rho_n}) = \diver(\rho_n\na^2\log\rho_n)$ and inequality \eqref{1.qineq} show that 
\begin{align}\label{3.qineq}
  \eps&\kappa\int_0^t\int_\Omega\na\rho_n\cdot\na
  \bigg(\frac{\Delta\sqrt{\rho_n}}{\sqrt{\rho_n}}\bigg)dxd\tau
  = \eps\kappa\int_0^t\int_\Omega\rho_n\na\bigg(
  \frac{\Delta\sqrt{\rho_n}}{\sqrt{\rho_n}}\bigg)
  \cdot\na\log\rho_n dxd\tau \\
  &= -\frac{\eps\kappa}{2}\int_0^t\int_\Omega
  \rho_n|\na^2\log\rho_n|^2 dxd\tau
  \le -C_K\eps\kappa\int_0^t\int_\Omega 
  \big(|\Delta\sqrt{\rho_n}|^2 + |\na\sqrt[4]{\rho_n}|^4\big)dxd\tau.
  \nonumber 
\end{align}
The chemotaxis force term is reformulated as in the proof of Lemma 3 of \cite{HuJu24}, inserting the mass equation in \eqref{3.rhoc},
\begin{align}\label{3.aux}
  \int_0^t\int_\Omega\rho_n\na c_n\cdot v_n dxd\tau
  &= -\int_0^t\int_\Omega c_n\diver(\rho_nv_n)dxd\tau
  = \int_0^t\int_\Omega c_n(\pa_\tau\rho_n - \eps\Delta\rho_n)dxd\tau \\
  &= \int_\Omega\rho_nc_n dx\bigg|_{\tau=0}^{\tau=t}
  - \int_0^t\int_\Omega\rho_n\pa_\tau c_n dxd\tau 
  - \eps\int_0^t\int_\Omega c_n\Delta \rho_n dxd\tau. \nonumber 
\end{align}
The second term on the right-hand side is rewritten by inserting the equation for $c_n$ in \eqref{3.rhoc}:
\begin{align}\label{3.aux2}
  -\int_0^t\int_\Omega\rho_n\pa_\tau c_n dxd\tau
  &= -\int_0^t\int_\Omega(\pa_\tau c_n - \Delta c_n + c_n)\pa_\tau c_n
  dxd\tau \\
  &= -\int_0^t\int_\Omega(\pa_\tau c_n)^2 dxd\tau
  - \frac12\int_\Omega(|\na c_n|^2 + c_n^2)dxd\tau
  \bigg|_{\tau=0}^{\tau=t} \nonumber 
\end{align}
Using this expression, the last term on the right-hand side of \eqref{3.aux} becomes
\begin{align*}
  -\eps&\int_0^t\int_\Omega c_n\Delta \rho_n dxd\tau
  = -\eps\int_0^t\int_\Omega\rho_n\Delta c_n dxd\tau
  = -\eps\int_0^t\int_\Omega\rho_n(\pa_\tau c_n + c_n-\rho_n)dxd\tau \\
  &= -\eps\int_0^t\int_\Omega(\pa_\tau c_n)^2 dxd\tau
  - \frac{\eps}{2}\int_\Omega(|\na c_n|^2 + c_n^2)dxd\tau
  \bigg|_{\tau=0}^{\tau=t}
  - \eps\int_0^t\int_\Omega \rho_n c_n dxd\tau \\
  &\phantom{xx}+ \eps\int_0^t\int_\Omega \rho_n^2 dxd\tau.
\end{align*}
We neglect the nonpositive terms to arrive at
\begin{align}\label{3.cdelta}
  -\eps\int_0^t\int_\Omega c_n\Delta \rho_n dxd\tau
  \le \frac{\eps}{2}\int_\Omega(|\na c^0|^2+(c^0)^2)dx
  + \eps\int_0^t\int_\Omega \rho_n^2 dxd\tau.
\end{align}
Then, together with \eqref{3.aux2}, we infer from \eqref{3.aux} that
\begin{align*}
  \int_0^t&\int_\Omega\rho_n\na c_n\cdot v_n dxd\tau
  \le \int_\Omega(\rho_nc_n)(t) dx
  - \int_0^t\int_\Omega(\pa_\tau c_n)^2 dxd\tau \\
  &- \frac{1}{2}\int_\Omega(|\na c_n(t)|^2 + c_n(t)^2)dx
  + \frac{1+\eps}{2}\int_\Omega(|\na c^0|^2 + (c^0)^2)dx
  + \eps\int_0^t\int_\Omega \rho_n^2 dxd\tau.
\end{align*}
Inserting \eqref{3.qineq} and the previous inequality into \eqref{3.aux0} finishes the proof.
\end{proof}

It remains to estimate the last two terms in \eqref{3.aei}. For this, we apply Lemma \ref{lem.aux}:
\begin{align}
  \int_\Omega(\rho_nc_n)(t)dx &\le C(\kappa) 
  + \frac{\kappa}{2}\int_\Omega|\na\sqrt{\rho_n(t)}|^2dx
  + \frac14\int_\Omega|\na c_n(t)|^2 dx, \label{3.rc} \\
  \eps\int_0^t\int_\Omega\rho_n^2 dxd\tau
  &\le C\eps T + \frac{C_K}{2}\eps\kappa
  \int_0^t\int_\Omega|\na\sqrt[4]{\rho_n}|^{4} dxd\tau. \label{3.r2}
\end{align}
The terms on the right-hand side of \eqref{3.rc} can be absorbed by the energy $E_K$ on the left-hand side of \eqref{3.aei}. The integral on the right-hand side of \eqref{3.r2} can be absorbed by the dissipation term $D_{ho}$. Taking into account that $\eps\le 1$, we have proved the following result.

\begin{lemma}[Modified energy inequality]\label{lem.mei}
Let $(\rho_n,c_n,v_n)$ be the local solution to \eqref{3.rhoc} and \eqref{3.vn}. Then, for $0<t\le T_n$,
\begin{align*}
  \widetilde{E}&((\rho_n,v_n,c_n)(t))
  + \int_0^t\big(D_K(\rho_n,v_n,c_n)
  + \widetilde{D}_{ho}(\rho_n,v_n,c_n)\big)d\tau \\
  &\le C(\kappa,T) + 2 E_K(\rho^0,v^0,c^0) 
  + E_{ho}(\rho^0),
\end{align*}
where
\begin{align*}
  \widetilde{E}(\rho_n,v_n,c_n)
  &= \int_\Omega\bigg(\frac{1}{2}\rho_n|v_n|^2
  + h(\rho_n) + \frac{\kappa}{2}|\na\sqrt{\rho_n}|^2 
  + \frac14(|\na c_n|^2+c_n^2) \\
  &\phantom{xx}+ \frac{\eta}{4}\rho_n^{-3}
  + \frac{\delta}{2}|\na\Delta^2\rho_n|^2\bigg)dx, \\
  \widetilde{D}_{ho}(\rho_n,v_n,c_n)
  &= \int_\Omega\bigg(\frac{4\eps}{\gamma}|\na\rho_n^{\gamma/2}|^2
  + \frac{C_K}{2}\eps\kappa\big(|\Delta\sqrt{\rho_n}|^2
  + |\na\sqrt[4]{\rho_n}|^4\big) \\
  &\phantom{xx}+ \mu|\Delta v_n|^2
  + \frac{4\eps \eta}{3}|\nabla\rho^{-3/2}|^2 
  + \eps\delta|\Delta^3\rho_n|^2
  \bigg)dx.
\end{align*}
\end{lemma}

The estimates of the previous lemma allow us to extend the local solution to a global one. 


\subsection{Limit $n\to\infty$}\label{sec.limitn}

The limit $n\to\infty$ can be performed as in \cite[Sec.~2]{VaYu16} except for the new chemotaxis terms. Lemma \ref{lem.mei} yields the following bounds uniform in $n$:
\begin{align*}
  \|\sqrt{\rho_n}v_n\|_{L^2(\Omega_T)}
  + \|\sqrt[4]{\rho_n}v_n\|_{L^4(\Omega_T)}
  + \|\sqrt{\rho_n}\mathbb{D}(v_n)\|_{L^2(\Omega_T)}
  + \sqrt{\delta}\|\rho_n\|_{L^\infty(0,T;H^5(\Omega))}
  \le C.
\end{align*}
Moreover, we deduce from the uniform bounds for $\sqrt{\rho_n}$ in $L^\infty(0,T;H^1(\Omega))\subset L^\infty(0,T;$ $L^6(\Omega))$ and $L^2(0,T;H^2(\Omega))$ as well as the Gagliardo--Nirenberg inequality that $(\sqrt{\rho_n})$ is bounded in $L^{10}(\Omega_T)$.
This bound does not arise from the higher-order approximation terms.
Consequently,
\begin{align*}
  \pa_t\rho_n &= -(4\na\sqrt[4]{\rho_n})\cdot(\sqrt[4]{\rho_n}v_n)
  \sqrt{\rho_n} - \sqrt{\rho_n}(\sqrt{\rho_n}\diver v_n) \\
  &\phantom{xx} + 2\sqrt{\eps\rho_n}
  (\sqrt\eps\Delta\sqrt{\rho_n})
  + 8\sqrt{\eps\rho_n}(\sqrt\eps|\na\sqrt[4]{\rho_n}|^2)
\end{align*}
is uniformly bounded in $L^{5/3}(\Omega_T)$. We deduce from the Aubin--Lions lemma that there exists a subsequence (not relabeled) such that, as $n\to\infty$,
\begin{align*}
  \rho_n\to\rho\quad\mbox{strongly in }L^q(\Omega_T)\quad
  \mbox{for all } q<10.
\end{align*}
The bounds for $(\pa_tc_n)$ in $L^2(\Omega_T)$ and $(c_n)$ in $L^2(0,T;H^1(\Omega))$ yield the existence of a subsequence such that
\begin{align*}
  c_n\to c_n\quad\mbox{strongly in }L^2(\Omega_T), \quad
  \na c_n\rightharpoonup\na c\quad\mbox{weakly in }L^2(\Omega_T).
\end{align*}
We conclude that
\begin{align*}
  \rho_n\na c_n\rightharpoonup \rho\na c\quad\mbox{weakly in }
  L^1(\Omega_T). 
\end{align*}
Together with the convergence results in \cite[Sec.~2]{VaYu16}, this allows us to perform the limit $n\to\infty$ in equations \eqref{1.rho} for the mass and concentration as well as in the momentum equation \eqref{3.vn}. 

It remains to perform the limit $n\to\infty$ in the approximate energy inequality \eqref{3.aei}, which is done by exploiting the weak lower semicontinuity of convex functions except for the integral $\int_\Omega\rho_n(t)c_n(t)dx$. Since $(\rho_n)$ is bounded in $L^\infty(0,T;H^5(\Omega))$, $(c_n)$ is bounded in $L^\infty(0,T;H^1(\Omega))$, and the embeddings $H^5(\Omega)\hookrightarrow L^\infty(\Omega)$, $H^1(\Omega)\hookrightarrow L^2(\Omega)$ are compact, the Aubin--Lions lemma implies that (up to a subsequence)
\begin{align*}
  \rho_n\to\rho_n&\quad\mbox{strongly in }
  C^0([0,T];L^\infty(\Omega)), \\
  \quad c_n\to c&\quad\mbox{strongly in }C^0([0,T];L^2(\Omega)).
\end{align*}
This shows that
\begin{align*}
  \int_\Omega\rho_n(t)c_n(t)dx \to \int_\Omega\rho(t)c(t)dx
  \quad\mbox{for }0\le t\le T.
\end{align*}

We summarize our results.

\begin{proposition}\label{prop.approx}
Let $T>0$ and let the initial data satisfy \eqref{1.init}. Then there exists a solution $(\sqrt\rho,\sqrt\rho v,c)$ to \eqref{3.rhoc}--\eqref{3.v} such that equations \eqref{3.rhoc} are satisfied a.e.\ and \eqref{3.v} is satisfied in the sense
\begin{align}\label{3.wv}
  -&\int_\Omega\rho^0 v^0\cdot\theta(0,\cdot)dx
  - \int_0^T\int_\Omega\rho v\cdot\pa_t\theta dxdt \\
  &= \int_0^T\int_\Omega\bigg(\rho v\otimes v:\na\theta
  + p(\rho)\diver\theta - \nu v\cdot\diver(\rho\mathbb{D}(\theta))
  + \rho\na c\cdot\theta - \frac{\rho v}{\zeta}\cdot \theta \bigg)dxdt 
  \nonumber \\
  &\phantom{xx}
  - \int_0^T\int_\Omega\bigg(\kappa\frac{\Delta\sqrt\rho}{\sqrt\rho}
  \diver(\rho\theta) + r_0 v\cdot\theta + r_1\rho|v|^2v\cdot\theta
  \bigg)dxdt \nonumber \\
  &\phantom{xx}
  - \int_0^T\int_\Omega\bigg(\eps (v\otimes\na\rho):\na\theta
  + \Delta v\cdot\Delta\theta + \nu\rho^{-3}\diver\theta
  + \delta\Delta^3\rho\Delta^2\diver(\rho\theta)\bigg)dxdt
  \nonumber 
\end{align}
for all $\theta\in L^\infty(0,T;H^4(\Omega;\R^3))\cap L^2(0,T;H^5(\Omega;\R^3))$ satisfying $\theta(T,\cdot)=0$. Furthermore, the approximate energy inequality \eqref{3.aei} holds with $(\rho_n,v_n,c_n)$ replaced by $(\rho,v,c)$. 
\end{proposition}

The weak formulation \eqref{3.vn} and consequently \eqref{3.wv} first hold for test functions $\theta\in C^1([0,T);X_n)$ for all $n\in\N$. A standard density argument extends the validity of \eqref{3.wv} to the broader set of test functions specified in Proposition \ref{prop.approx}. 


\subsection{BD entropy inequality}

We derive the BD entropy inequality for the approximate system \eqref{3.rhoc}--\eqref{3.v} yielding estimates uniform in $(\delta,\eps,\eta,\mu)\to 0$. 

\begin{lemma}[BD entropy inequality]\label{lem.bdent}
Let $(\sqrt\rho,\sqrt\rho v,c)$ be the solution to \eqref{3.rhoc}--\eqref{3.v} on $[0,T]$ constructed in Proposition \ref{prop.approx}. Then there exist constants $M=M(\nu,r_1)>2$,  $C_1(\delta,\eta,\kappa,r_0,$ $r_1,T)>0$, $C_2(\kappa,T)>0$, and $C_3>0$ (independent of all approximation parameters) such that
\begin{align}\label{3.bdineq}
  &\frac{M-2}{2}E_K((\rho,v,c)(t)) 
  + M E_{ho}(\rho(t)) 
  + \frac{\nu^2}{2}\int_\Omega |\na\sqrt{\rho(t)}|^2 dx
  + \nu r_0\int_\Omega(\log\rho(t))_- dx \\
  &\phantom{xx}+ \frac{M}{2}\int_0^t\int_\Omega\bigg(
  \frac{\rho|v|^2}{\zeta} + r_0|v|^2 + r_1\rho|v|^4\bigg)dxd\tau 
  \nonumber \\
  &\phantom{xx}+ (M-1)\int_0^t\int_\Omega\big(\nu\rho|\mathbb{D}(v)|^2
  + |\pa_\tau c|^2\big)dxd\tau \nonumber \\
  &\phantom{xxx}+ \int_0^t\int_\Omega\bigg(
  \frac{\nu}{4}\rho|\na v-\na v^T|^2
  + \frac{4\nu}{\gamma}|\na\rho^{\gamma/2}|^2
  + C_K\nu\kappa\big(|\na\sqrt[4]{\rho}|^4
  + |\Delta\sqrt\rho|^2\big)\bigg)dxd\tau \nonumber \\
  &\phantom{xxx}+ \int_0^t\int_\Omega\bigg(
  \nu\delta|\Delta^3\rho|^2 
  + \frac{4\eta\nu}{3}|\nabla\rho^{-3/2}|^2
  + \mu M|\Delta v|^2 + \eps\nu^2\frac{|\Delta\rho|^2}{\rho}
  \bigg)dxd\tau \nonumber \\
  &\phantom{x}\le \bigg(\eps + \sqrt\mu + \frac{\eps}{\sqrt\mu}\bigg)
  C_1(\delta,\eta,\kappa,r_0,r_1,T) + C_2(\kappa,T) 
  + C_3\nu r_1 \nonumber 
\end{align}
for $0<t<T$, where the energies $E_K$ and $E_{ho}$ are defined in \eqref{3.EK} and \eqref{3.Eho}, respectively. If $\gamma > 4/3$, the constant $C_2=C_2(T)$ does not depend on $\kappa$.
\end{lemma}

Observe that  the right-hand side of \eqref{3.bdineq} becomes independent of $(\kappa,r_0,r_1)$ after having passed to the limit $(\eps,\mu,\eps/\sqrt\mu)\to 0$ in the case $\gamma>4/3$.

\begin{proof}
The proof is a lengthy computation split into several steps. 

{\em Step 1: Intermediate BD entropy inequality.} Similarly as in \cite[Lemma 12]{Zat12}, we compute the derivatives
\begin{align*}
  \int_0^t\int_\Omega\pa_\tau\bigg(\frac{\nu^2}{2}\rho|\na\log\rho|^2
  \bigg)dxd\tau,
  \quad \nu\int_0^t\int_\Omega\pa_\tau(\rho v\cdot\na\log\rho)dxd\tau
\end{align*}
and sum the resulting expressions. After some calculations, we obtain
\begin{align*}
  &\int_\Omega\bigg(\frac{\nu^2}{2}\rho|\na\log\rho|^2
  + \nu\rho v\cdot\na\log\rho\bigg)dx\bigg|_{\tau=0}^{\tau=t}
  = \nu\int_0^t\int_\Omega\bigg(\rho|\mathbb{D}(v)|^2
  + \kappa\na\rho\cdot\na\bigg(\frac{\Delta\sqrt\rho}{\sqrt\rho}\bigg) \\
  &+ \na\rho\cdot\na c - \frac14\rho|\na v-\na v^T|^2
  - \frac{4}{\gamma}|\na\rho^{\gamma/2}|^2
  - \frac{\na\rho\cdot v}{\zeta}\bigg)dxd\tau \\
  &+ \eps\nu\int_0^t\int_\Omega\bigg(
  \frac{\nu}{2}\Delta\rho|\na\log\rho|^2 
  + (\na\rho\otimes\na\log\rho):\na v - \Delta\rho\diver v
  - \nu\frac{|\Delta\rho|^2}{\rho}\bigg)dxd\tau \\
  &- \nu\int_0^t\int_\Omega\bigg(\mu\na\log\rho\cdot\Delta^2 v 
  + r_0 \na\log\rho\cdot v + r_1 \na\rho\cdot v|v|^2
  + \frac{4\eta}{3}|\nabla \rho^{-3/2}|^2 
  + \delta|\Delta^3\rho|^2\bigg)dxd\tau.
\end{align*}
We add the energy inequality \eqref{3.aei} (with $(\rho_n,v_n,c_n)$ replaced by $(\rho,v,c)$), multiplied by $M$, to the previous equation, which leads to
\begin{align}\label{3.bdent}
  \frac12&\int_\Omega\rho|v+\nu\na\log\rho|^2(t) dx
  + (M-1)E_K((\rho,v,c)(t)) + ME_{ho}(\rho(t)) \\
  &\phantom{xx}+ M\int_0^t\int_\Omega\bigg(\frac{\rho|v|^2}{\zeta}
  + r_0|v|^2 + r_1\rho |v|^4\bigg)dxd\tau
  + M\int_0^t\int_\Omega|\pa_\tau c|^2 dxd\tau \nonumber \\
  &\phantom{xx}+ \nu\int_0^t\int_\Omega\bigg(\frac14\rho|\na v-\na v^T|^2
  + \frac{4}{\gamma}|\na\rho^{\gamma/2}|^2
  + C_K\kappa\big(|\Delta\sqrt\rho|^2
  + |\na\sqrt[4]{\rho}|^4\big)\bigg)dxd\tau \nonumber \\
  &\phantom{xx}+ \nu(M-1)\int_0^t\int_\Omega\rho
  |\mathbb{D}(v)|^2 dxd\tau 
  + \int_0^t\int_\Omega\bigg(\frac{4\eta\nu}{3}
  |\nabla\rho^{-3/2}|^2
  + \nu\delta|\Delta^3\rho|^2\bigg)dxd\tau \nonumber \\
  &\phantom{xx}+ \eps\int_0^t\int_\Omega\bigg(
  \frac{4}{\gamma}M|\na\rho^{\gamma/2}|^2
  + C_K\kappa M\big(|\Delta\sqrt\rho|^2
  + |\na\sqrt[4]{\rho}|^4\big) + \mu|\Delta v|^2 \nonumber \\
  &\phantom{xx}+ \frac{4\eta M}{3}|\nabla\rho^{-3/2}|^2 
  + \delta M|\Delta^3\rho|^2
  + \nu^2\frac{|\Delta\rho|^2}{\rho}\bigg)dxd\tau \nonumber \\
  &\le I_0 + J_1 + \cdots + J_5 + K_1 + \cdots + K_5, \nonumber 
\end{align}
where
\begin{align*}
  I_0 &= M(1+\eps)E_K(\rho^0,v^0,c^0) + ME_{ho}(\rho^0)
  + \int_\Omega\bigg(\frac{\nu^2}{2}\rho^0|\na\log\rho^0|^2
  + \nu\rho^0v^0\cdot\na\log\rho^0\bigg)dx
\end{align*}
as well as
\begin{align*}
  J_1 &= \frac{\nu^2}{2}\eps\int_0^t\int_\Omega\Delta\rho|\na\log\rho|^2
  dxd\tau, &
  J_2 &= -\nu\mu\int_0^t\int_\Omega\na\Delta\log\rho\cdot
  \Delta vdxd\tau, \\
  J_3 &= \eps\nu\int_0^t\int_\Omega(\na\rho\otimes\na\log\rho):\na v 
  dxd\tau, &
  J_4 &= -\eps\nu\int_0^t\int_\Omega\Delta\rho\diver v dxd\tau, \\
  J_5 &= \eps M\int_0^t\int_\Omega\rho^2 dxd\tau, & 
  K_1 &= \nu\int_0^t\int_\Omega\na\rho\cdot\na c dxd\tau, \\
  K_2 &= -\nu\int_0^t\int_\Omega\frac{\na\rho\cdot v}{\zeta}dxd\tau, &
  K_3 &= -\nu r_0\int_0^t\int_\Omega\na\log\rho\cdot vdxd\tau, \\
  K_4 &= -\nu r_1\int_0^t\int_\Omega\na\rho\cdot v|v|^2 dxd\tau, &
  K_5 &= M\int_\Omega\rho(t)c(t)dx.  
\end{align*}
Observe that $J_i$ depends on the parameters $(\eps,\mu)$, which will tend to zero first, while $K_3$ and $K_4$ depend on $(r_0,r_1)$, which are fixed at this point. 

{\em Step 2: Positivity of $\rho$.} We claim that $\rho$ is positive a.e. We know from the modified energy inequality in Lemma \ref{lem.mei} that
\begin{align}\label{3.estrho}
  \eta^{1/3}\|\rho^{-1}\|_{L^\infty(0,T;L^3(\Omega))}
  + \sqrt{\delta}\|\rho\|_{L^\infty(0,T;H^5(\Omega))} \le C(\kappa).
\end{align}
Inequality \eqref{3.pos} then yields
\begin{align}\label{3.rho-1}
  \|\rho^{-1}\|_{L^\infty(\Omega_T)}
  \le C(1 + C\eta^{-1/3})^3(1+C\delta^{-1/2})^2
  \le C(1+\eta^{-1})(1+\delta^{-1}) =: C_{\delta,\eta}.
\end{align}
Thus, $\rho^{-1}$ is bounded a.e.\ in $\Omega_T$, which proves the claim. 

{\em Step 3: Estimation of $J_i$.} We start with the term $J_1$ and use \eqref{3.estrho}:
\begin{align*}
  |J_1| \le \eps\nu^2\|\rho^{-1}\|_{L^\infty(\Omega_T)}^2
  \|\rho\|_{L^2(0,T;H^2(\Omega))}\|\na\rho\|_{L^4(\Omega_T)}^2
  \le C\eps\delta^{-1/2}\delta^{-1}C_{\delta,\eta}^2.
\end{align*}
It holds that
\begin{align*}
  |J_2|\le \nu\mu\|\Delta\na\log\rho\|_{L^2(\Omega_T)}
  \|\Delta v\|_{L^2(\Omega_T)}.
\end{align*} 
We calculate $\na\Delta\log\rho$. Let $\pa_i=\pa/\pa x_i$. Then
\begin{align*}
  \pa_i\pa_j^2\log\rho = \frac{\pa_i\pa_j^2\rho}{\rho}
  - \frac{\pa_j^2\rho\pa_i\rho}{\rho^2}
  - 2\frac{\pa_i\pa_j\rho\pa_j\rho}{\rho^2}
  + 2\frac{\pa_i\rho(\pa_j\rho)^2}{\rho^3}.
\end{align*}
The denominators are bounded in $L^\infty(\Omega_T)$ because of \eqref{3.estrho}, so it remains to estimate the nominators in the $L^2(\Omega_T)$ norm. We deduce from \eqref{3.estrho} that 
\begin{align*}
  \|\pa_i\pa_j^2\rho\|_{L^2(\Omega_T)} 
  &\le \|\rho\|_{L^2(0,T;H^3(\Omega))} \le C(\kappa)\delta^{-1/2}, \\
  \|\pa_i\pa_j\rho\pa_i\rho\|_{L^2(\Omega_T)}
  &\le \|\rho\|_{L^\infty(0,T;W^{2,4}(\Omega))}
  \|\rho\|_{L^2(0,T;W^{1,4}(\Omega))} \\
  &\le C\|\rho\|_{L^\infty(0,T;H^{3}(\Omega))}
  \|\rho\|_{L^2(0,T;H^{2}(\Omega))} \le C(\kappa)\delta^{-1}, \\
  \|\pa_i\rho(\pa_j\rho)^2\|_{L^2(\Omega_T)}
  &\le \|\rho\|_{L^6(0,T;W^{1,6}(\Omega))}^3
  \le C\|\rho\|_{L^\infty(0,T;H^2(\Omega))}^3 
  \le C(\kappa)\delta^{-3/2}.
\end{align*}
Together with the bound $\sqrt\mu\|\Delta v\|_{L^2(\Omega_)}\le C(\kappa)$ from Lemma \ref{lem.mei}, we obtain 
\begin{align*}
  |J_2|\le C(\kappa)(1+C_{\delta,\eta}^3)(1+\delta^{-3/2})\mu^{1/2}.
\end{align*}
It follows from Lemma \ref{lem.mei} (modified energy inequality) that
\begin{align*}
  |J_3| &\le 16\eps\nu\int_0^t\int_\Omega
  \big|\sqrt\rho(\na\sqrt[4]{\rho}
  \otimes\na\sqrt[4]{\rho}):\mathbb{D}(v)\big|dxd\tau \\
  &\le C\eps\|\sqrt\rho\mathbb{D}(v)\|_{L^2(\Omega_T)}
  \|\na\sqrt[4]{\rho}\|_{L^4(\Omega_T)}^2
  \le C(\kappa)\eps\cdot\eps^{-1/2} = C(\kappa)\eps^{1/2}.
\end{align*}
Integrating by parts in $J_4$ several times, 
\begin{align*}
  \int_\Omega\Delta\rho\diver v dx
  = \int_\Omega\sum_{i,j=1}^d
  \frac{\pa^2\rho}{\pa x_i^2}\frac{\pa v_j}{\pa x_j}dx
  = -\int_\Omega \sum_{i,j=1}^d
  \frac{\pa\rho}{\pa x_j}\frac{\pa^2 v_j}{\pa x_i^2}dx
  = -\int_\Omega\na\rho\cdot\Delta v dx,
\end{align*}
we find that
\begin{align*}
  |J_4| \le C\eps\|\rho\|_{L^2(0,T;H^1(\Omega))}
  \|\Delta v\|_{L^2(\Omega_T)}\le C(\kappa)\eps(\delta\mu)^{-1/2}.
\end{align*}
Finally, because of Lemma \ref{lem.aux},
\begin{align*}
  |J_5|\le C\eps MT + \eps C_K \kappa M\int_0^t\int_\Omega
  |\na\sqrt[4]{\rho}|^4 dxd\tau,
\end{align*}
and the integral can be absorbed by the left-hand side of \eqref{3.bdent}. Summarizing, we find that
\begin{align}\label{3.J}
  |J_1|+\cdots+|J_5|\le C(\delta,\eta,\kappa)
  \bigg(\sqrt{\eps} + \sqrt{\mu} + \frac{\eps}{\sqrt\mu}\bigg) 
  + \eps C_K \kappa M\int_0^t\int_\Omega|\na\sqrt[4]{\rho}|^4 dxd\tau.
\end{align}

{\em Step 4: Estimation of $K_i$.} We conclude from \eqref{3.cdelta} that
\begin{align*}
  K_1 = -\nu\int_0^t\int_\Omega\Delta\rho c dx d\tau
  \le \frac{\nu}{2}\int_\Omega(|\na c^0|^2 + (c^0)^2)dx
  + \nu\int_0^t\int_\Omega \rho^2 dxd\tau.
\end{align*}
We use Lemma \ref{lem.aux} to estimate the last integral on the right-hand side. Here, we distinguish between the cases $\gamma>1$ and $\gamma>4/3$. If $\gamma>1$, we have for any $\sigma>0$,
\begin{align*}
  K_1 \le C(c^0,\nu,T) + C(\sigma) 
  + \sigma\int_0^t\int_\Omega |\na\sqrt[4]{\rho}|^4 dxd\tau,
\end{align*}
and the integral on the right-hand side can be absorbed by the left-hand side of \eqref{3.bdent} after choosing $\sigma<C_K\nu\kappa/2$. Here, the constant $C(\sigma)$ depends on $\kappa$. This dependency can be avoided by choosing $\gamma>4/3$. Indeed, we find that
\begin{align*}
  K_1 \le C(c^0,\nu,T) + \frac{2\nu}{\gamma}
  \int_0^t\int_\Omega|\na\rho^{\gamma/2}|^2dxd\tau,
\end{align*}
and the last integral can be absorbed by the left-hand side of \eqref{3.bdent} without any dependency on $\kappa$.

Next, it follows from Young's inequality that
\begin{align*}
  K_2 = 2\frac{\nu}{\zeta}\int_0^t\int_\Omega\sqrt{\rho} v
  \cdot\na\sqrt\rho dxd\tau
  \le \frac{\nu}{\zeta}\int_0^t\int_\Omega|\sqrt\rho v|^2 dxd\tau
  + \frac{\nu}{\zeta}\int_0^t\int_\Omega|\na\sqrt\rho|^2 dxd\tau.
\end{align*}
The first term on the right-hand side can be absorbed by the left-hand side of \eqref{3.bdent}, while the second term will be estimated later.

The estimate of $K_3$ is more involved. We deduce from the mass equation that $\na\rho\cdot v = \eps\Delta\rho - \pa_t\rho - \rho\diver v$ and therefore, using $\int_\Omega\diver v dx=0$,
\begin{align*}
  K_3 &= -\nu r_0\int_0^t\int_\Omega\frac{\na\rho\cdot v}{\rho}dxd\tau
  = \nu r_0\int_0^t\int_\Omega\frac{1}{\rho}(\rho\diver v + \pa_t\rho 
  - \eps\Delta\rho)dxd\tau \\
  &= \nu r_0\int_0^t\int_\Omega\bigg(\pa_t\log\rho
  - \eps\frac{\Delta\rho}{\rho}\bigg)dxd\tau \\
  &\le \nu r_0\int_\Omega\big(\log\rho(t) - \log\rho^0\big)dx
  + \eps\nu r_0\|\rho\|_{L^2(0,T;H^2(\Omega))}
  \|\rho^{-1}\|_{L^\infty(\Omega_T)}.
\end{align*} 
It follows from $\log\rho(t)=(\log\rho(t))_+ - (\log\rho(t))_-
\le \rho(t)-(\log\rho(t))_-$ and mass conservation that
\begin{align*}
  K_3 &\le \nu r_0\int_\Omega\rho(t)dx 
  - \nu r_0\int_\Omega(\log\rho(t))_-dx
  - \nu r_0\int_\Omega\log\rho^0 dx \\
  &\phantom{xx}+ \eps\nu r_0\|\rho\|_{L^2(0,T;H^2(\Omega))}
  \|\rho^{-1}\|_{L^\infty(\Omega_T)} \\
  &\le C(\rho^0) - \nu r_0\int_\Omega(\log\rho(t))_-dx
  + C(\kappa)C_{\delta,\eta}\delta^{-1/2}\eps \\
  &\le C(\rho^0) + C(\delta,\eta,\kappa)\eps
  - \nu r_0\int_\Omega(\log\rho(t))_-dx,
\end{align*}
where we used \eqref{3.rho-1} in the next-to-last step. By Young's inequality,
\begin{align*}
  K_4 &= \nu r_1\int_0^t\int_\Omega \diver(|v|^2 v)\rho dxd\tau
  =  \nu r_1\int_0^t\int_\Omega\big(2\rho v\cdot\na v\cdot v
  + \rho|v|^2\diver v\big)dxd\tau \\
  &=  \nu r_1\int_0^t\int_\Omega\bigg(2\rho\sum_{i,j=1}^d 
  \mathbb{D}(v)_{ij}v_iv_j + \rho|v|^2\diver v\bigg)dxd\tau \\
  &\le  \nu r_1\int_0^t\int_\Omega\rho\sum_{i,j=1}^d(|v_i|^2|v_j|^2
  + |\mathbb{D}(v)_{ij}|^2)dxd\tau 
  + \frac{\nu}{2}r_1\int_0^t\int_\Omega\rho(|v|^4 + |\diver v|^2)dxd\tau.
\end{align*}
Then, since $|\diver v|^2\le 3|\mathbb{D}(v)|^2$,
\begin{align*}
  K_4 \le  \nu r_1\int_0^t\int_\Omega\bigg(\frac32\rho|v|^4 
  + \frac52\rho|\mathbb{D}(v)|^2\bigg)dxd\tau.
\end{align*}
The right-hand side can be absorbed by the left-hand side of \eqref{3.bdent} choosing $M$ sufficiently large (and depending on $\nu$ and $r_1$). Finally, by Lemma \ref{lem.aux}, for $\sigma_1=\nu/(2M)$ and $\sigma_2=1/4$,
\begin{align*}
  K_5 \le C(\nu,M)
  + \frac{\nu^2}{2}\int_\Omega|\na\sqrt{\rho(t)}|^2 dx
  + \frac{M}{4}\int_\Omega|\na c(t)|^2 dx.
\end{align*}
The last term can be absorbed by the left-hand side of \eqref{3.bdent}. We summarize:
\begin{align}\label{3.K}
  K_1&+\cdots+K_5 \le K_1 + C(\kappa,\nu,\rho^0) 
  + C(\delta,\eta,\kappa)\eps
  + \frac{\nu}{\zeta}\int_0^t\int_\Omega\rho|v|^2 dxd\tau \\
  &+ \frac{\nu}{\zeta}\int_0^t\int_\Omega|\na\sqrt\rho|^2
  dxd\tau + \nu r_1\int_0^t\int_\Omega\bigg(\frac32\rho|v|^4 
  + \frac52\rho|\mathbb{D}(v)|^2\bigg)dxd\tau \nonumber \\
  &+ \frac{\nu^2}{2}\int_\Omega|\na\sqrt{\rho(t)}|^2 dx
  + \frac{M}{4}\int_\Omega|\na c(t)|^2 dx
  - \nu r_0\int_\Omega(\log\rho(t))_-dx. \nonumber 
\end{align}

{\em Step 5: End of the proof.} The last integral of \eqref{3.J} as well as the first, third, and fifth integral of \eqref{3.K} can be absorbed by the corresponding terms on the left-hand side of \eqref{3.bdent} by choosing $M$ sufficiently large. For the fourth integral on the right-hand side, we compute
\begin{align*}
  \nu^2|\na\sqrt\rho|^2 
  = \frac14\nu^2\rho|\na\log\rho|^2 \le \frac12\rho|v+\nu\na\log\rho|^2
  + \frac12\rho|v|^2,
\end{align*}
which yields
\begin{align*}
  \frac{\nu^2}{2}\int_\Omega|\na\sqrt\rho|^2 dx
  \le \frac{1}{4}\int_\Omega\rho|v+\nu\na\log\rho|^2 dx 
  + \frac{1}{4}\int_\Omega \rho|v|^2 dx.
\end{align*}
The right-hand side can be absorbed by the left-hand side of \eqref{3.bdent}. Finally, the second integral on the right-hand side of \eqref{3.K} can be controlled by applying Gronwall's inequality after inserting estimates \eqref{3.J} and \eqref{3.K} into \eqref{3.bdent}. This finishes the proof.
\end{proof}


\subsection{Limit $(\delta,\eps,\eta,\mu)\to 0$}

In view of the estimates obtained from the BD entropy inequality of Lemma \ref{lem.bdent}, we can pass to the limit  $(\delta,\eps,\eta,\mu)\to 0$ similarly as in \cite[Sec.~3.2--3.3]{VaYu16}. More precisely, we pass to the limit $(\eps,\mu)\to 0$ such that $\eps/\sqrt\mu\to 0$. By Lemma \ref{lem.bdent}, this eliminates the dependence of the right-hand side on $(\delta,\eta)$. Hence, we obtain estimates uniform in these parameters and we can perform the limit $(\delta,\eta)\to 0$. These limits are treated exactly as in \cite[Sec.~3.2--3.3]{VaYu16}. The chemotaxis term does not introduce additional complications and can be handled by the arguments used in the limit $n\to\infty$ (see Section \ref{sec.limitn}). The limit $(\delta,\eps,\eta,\mu)\to 0$ in the BD entropy inequality yields \eqref{3.bdineq} with the first term on the right-hand side removed and with $(\delta,\eps,\eta,\mu)=0$.

Our arguments are valid for smooth initial data. The final step to complete the proof is to extend the result to less regular initial data. This is done in the usual way by constructing a sequence of smooth functions approximating the initial conditions; see, e.g., \cite[p.~1041]{Jue10}. This ends the proof of Theorem \ref{thm.nske}. 


\section{Proof of Theorem \ref{thm.nse}}\label{sec.nse}

We wish to pass to the limit $(\kappa,r_0,r_1)\to 0$ in \eqref{1a.rhoc}--\eqref{1a.v}. To this end, we proceed as in \cite{LaVa18}. First, we prove that the weak solution constructed in Theorem \ref{thm.nske} is also a renormalized solution in the sense specified in Definition \ref{def.weak}. Then we show that any renormalized solution is in fact a weak solution. Finally, we perform the limit $(\kappa,r_0,r_1)\to 0$.

\subsection{From weak to renormalized solutions}

The task is to reformulate the momentum equation \eqref{1a.v} for renormalized solutions. Let $(\sqrt\rho,\sqrt\rho v,c)$ be a weak solution to \eqref{1a.rhoc}--\eqref{1a.v}. The idea of renormalized solutions is to multiply the momentum equation by $\varphi'(v)$, where $\varphi$ is some smooth function. As indicated in \cite[Sec.~3]{LaVa18}, the regularity of the weak solution is not sufficient for this procedure and we need to regularize. For this,
let $\eta_\eps$ be a standard space-time mollifier on $(0,T)\times \mathbb{T}^3$, periodic in the spatial variables. (Since we passed to the limit $(\eps,\eta)\to 0$ in the previous section, these letters are reused here without causing any notation conflict.) The mollification of a function $g$ is denoted by
\begin{align*}
  \overline{g}^\eps(t,x) = (\eta_\eps * g)(t,x), \quad
  (t,x)\in(0,T)\times\mathbb{T}^3.
\end{align*}
Next, we define the continuous cutoff function
\begin{align*}
  \phi_m(y) = \begin{cases}
  0 &\mbox{for }0\le y\le 1/(2m), \\
  2my-1 &\mbox{for }1/(2m)\le y\le 1/m, \\
  1 &\mbox{for }1/m\le y\le m, \\
  2-y/m &\mbox{for }m\le y\le 2m, \\
  0 &\mbox{for }y\ge 2m,
  \end{cases}
\end{align*}
which vanishes close to zero and for large values of $y$ and equals one in the interval $[1/m,m]$. It holds that $\phi_m(y)\to 1$ as $m\to\infty$ for any fixed $y>0$. The cutoff function is used to define the approximate velocity $\widetilde{v}_m:=\phi_m(\rho)v$. We summarize the regularity obtained for the weak solution. Recall the notation \eqref{1.Skappa} for $\sqrt\rho\mathbb{S}_\kappa$ and \eqref{1.Tnu} for $\sqrt\rho\mathbb{T}_\nu$ as well as $\mathbb{S}_\nu = \frac12(\mathbb{T}_\nu + \mathbb{T}_\nu^T)$.

\begin{lemma}
Let $(\sqrt\rho,\sqrt\rho v,c)$ be a weak solution to \eqref{1a.rhoc}--\eqref{1a.v} constructed in Section \ref{sec.nse}. Then
\begin{align*}
  &\rho\in L^5(\Omega_T), &&\na\rho^\gamma\in L^{5/4}(\Omega_T;\R^3),
  && \na\phi_m(\rho)\in L^4(\Omega_T;\R^3), \\
  &\pa_t\phi_m(\rho)\in L^2(\Omega_T), 
  && v\in L^2(\Omega_T;\R^3), && \rho v\in L^{5/2}(\Omega_T;\R^3), \\
  & \rho|v|^2 v\in L^{5/4}(\Omega_T;\R^3), && 
  \rho|v|^2\in L^{5/3}(\Omega_T), 
  && \sqrt\rho|\mathbb{S}_\kappa|\in L^{5/3}(\Omega_T), \\
  &\sqrt\rho|\mathbb{S}_\nu|\in L^{5/3}(\Omega_T), &&
  \na c\in L^2(\Omega_T;\R^3), && \rho\na c\in L^{10/7}(\Omega_T;\R^3).
\end{align*}
\end{lemma}

\begin{proof}
The proof can be found in \cite[Lemma 3.3]{LaVa18}. It is based on the modified energy and the BD entropy inequalities, except for the chemotaxis term: It follows from $\rho\in L^5(\Omega_T)$ and $\na c\in L^2(\Omega_T;\R^3)$ that $\rho\na c\in L^{10/7}(\Omega_T;\R^3)$.
\end{proof}

We wish to derive the momentum equation satisfied by $\widetilde{v}_m=\phi_m(\rho)v$. As in \cite[Sec.~3.1]{LaVa18}, the idea is to use the mollified test function $\overline{\psi\phi_m(\rho)}^\eps$ in the momentum equation \eqref{1.wv}, where $\psi\in C_0^\infty(\Omega_T;\R^3)$ is a vector-valued function. Performing the limit $\eps\to 0$ in the resulting equation and using the commutator estimates of DiPerna and Lions \cite[Lemma 3.2]{LaVa18} leads to the following equation:
\begin{align}\label{4.aux0}
  0 &= \int_0^T\int_\Omega\big\{\rho\phi_m(\rho)v\cdot\pa_t\psi
  + \rho\phi_m(\rho)(v\otimes v):\na\psi 
  - \phi_m(\rho)\na p(\rho)\cdot\psi \\
  &\phantom{xx}
  - \phi_m(\rho)(\sqrt{\nu\rho}\mathbb{S}_\nu
  + \sqrt{\kappa\rho}\mathbb{S}_\kappa):\na\psi
  - (\sqrt{\nu\rho}\mathbb{S}_\nu
  + \sqrt{\kappa\rho}\mathbb{S}_\kappa):(\na\phi_m(\rho)\otimes\psi) 
  \nonumber \\
  &\phantom{xx}+ \rho\phi_m(\rho)\na c\cdot\psi
  - \phi_m(\rho)\frac{\rho v}{\zeta}\cdot\psi
  + \rho\pa_t\phi_m(\rho)v\cdot\psi
  + \rho (v\otimes v):(\na\phi_m(\rho)\otimes\psi) \nonumber \\
  &\phantom{xx}- r_0\phi_m(\rho)v\cdot\psi 
  - r_1\rho\phi_m(\rho)|v|^2v\cdot\psi\big\}dxdt. \nonumber 
\end{align}
We want to replace the term $\rho\pa_t\phi_m(\rho)v\cdot\psi + \rho (v\otimes v):(\na\phi_m(\rho)\otimes\psi)$ in the third line with the help of the mass equation. For this, we use \cite[Lemma 3.1]{LaVa18} to treat the mollified quantities, in particular
\begin{align}\label{4.gh}
  \int_0^\infty\int_\Omega \overline{g}^\eps h dxdt 
  = \int_0^\infty\int_\Omega g\overline{h}^\eps dxdt, \quad
  \lim_{\eps\to 0}\overline{\na(gh)}^\eps
  = \lim_{\eps\to 0}\na(g\overline{h}^\eps)
\end{align}
for suitable functions $g$ and $h$, where the limit holds in the sense of $L^q(\Omega_T)$ for any $q<\infty$. It follows from the test function $\overline{\phi'_m(\overline{\rho}^\eps)\rho\psi_i v_j}^\eps$ in \eqref{1.wrho}  (with $i,j=1,2,3$) and integration by parts that 
\begin{align}\label{4.aux}
  0 &= -\int_0^T\int_\Omega\big(
  \rho\pa_t\overline{\phi'_m(\overline{\rho}^\eps)\rho\psi_i v_j}^\eps
  + \rho v\cdot\na\overline{\phi'_m(\overline{\rho}^\eps)
  \rho\psi_i v_j}^\eps\big)dxdt \\
  &= \int_0^T\int_\Omega\big(\pa_t\overline{\rho}^\eps
  \phi'_m(\overline{\rho}^\eps)
  \rho\psi_i v_j + \overline{\diver(\rho v)}^\eps  
  \phi'_m(\overline{\rho}^\eps)\rho\psi_i v_j\big)dxdt. \nonumber 
\end{align}
Taking into account definition \eqref{1.Tnu} of $\mathbb{T}_\nu$, we have
\begin{align*}
  \sqrt{\nu\rho}\operatorname{tr}\mathbb{T}_\nu
  = \nu\diver(\rho v) - 2\nu\sqrt\rho v\cdot\na\sqrt\rho,
\end{align*}
such that \eqref{4.aux} becomes
\begin{align}\label{4.aux2}
  0 = \int_0^T\int_\Omega\bigg\{\pa_t\overline{\rho}^\eps
  \phi'_m(\overline{\rho}^\eps)
  \rho\psi_i v_j + \bigg(\overline{\sqrt{\frac{\rho}{\nu}}
  \operatorname{tr}\mathbb{T}_\nu}^\eps  
  + 2\overline{\sqrt\rho v\cdot\na\sqrt\rho}^\eps\bigg)
  \phi'_m(\overline{\rho}^\eps)\rho\psi_i v_j\bigg\}dxdt.
\end{align}
The BD entropy inequality in Lemma \ref{lem.bdent} shows that $\sqrt\rho v\cdot\na\sqrt\rho\phi'_m(\rho)\in L^{4/3}(\Omega_T)$ and $\phi'_m(\rho)\rho v_j\in L^4(\Omega_T)$. Therefore, we can pass to the limit $\eps\to 0$ in \eqref{4.aux2} to conclude that 
\begin{align*}
  0 &= \int_0^T\int_\Omega\bigg\{\pa_t\rho \phi'_m(\rho)\rho\psi_i v_j
  + \bigg(\sqrt{\frac{\rho}{\nu}}\operatorname{tr}\mathbb{T}_\nu  
  + 2\sqrt\rho v\cdot\na\sqrt\rho\bigg)\phi'_m(\rho)\rho \psi_i v_j
  \bigg\}dxdt \\
  &= \int_0^T\int_\Omega\bigg(\pa_t\phi_m(\rho)\rho\psi_i v_j
  + v\cdot\na\phi_m(\rho)\rho\psi_iv_j
  + \sqrt{\frac{\rho}{\nu}}\operatorname{tr}\mathbb{T}_\nu 
  \phi'_m(\rho)\rho \psi_i v_j\bigg)dxdt.
\end{align*}
Thus, replacing the term $\rho\pa_t\phi_m(\rho)v\cdot\psi + \rho (v\otimes v):(\na\phi_m(\rho)\otimes\psi)$ in \eqref{4.aux0} and using the definition $\widetilde{v}_m=\phi_m(\rho)v$, we obtain
\begin{align} 
  0 &= \int_0^T\int_\Omega\bigg\{
  \rho\widetilde{v}_m\cdot\pa_t\psi 
  + \rho (v\otimes\widetilde{v}_m):\na\psi 
  - \phi_m(\rho)\na p(\rho)\cdot\psi
  + \rho\phi_m(\rho)\na c\cdot\psi \nonumber \\
  &\phantom{xx}- \phi_m(\rho)(\sqrt{\nu\rho}\mathbb{S}_\nu
  + \sqrt{\kappa\rho}\mathbb{S}_\kappa):\na\psi
  - (\sqrt{\nu\rho}\mathbb{S}_\nu + \sqrt{\kappa\rho}\mathbb{S}_\kappa)
  :(\na\phi_m(\rho)\otimes\psi) \label{4.aux3} \\
  &\phantom{xx}
  - \frac{\rho\widetilde{v}_m}{\zeta}\cdot\psi
  - \sqrt{\frac{\rho}{\nu}}\operatorname{tr}\mathbb{T}_\nu
  \phi'_m(\rho)\rho v\cdot\psi - r_0\widetilde{v}_m\cdot\psi
  - r_1\rho|v|^2\widetilde{v}_m\cdot\psi\bigg\}dxdt. \nonumber 
\end{align}
We choose the test function $\psi=\overline{\phi\varphi'(\overline{\widetilde{v}_m}^\eps)}^\eps$ for some $\phi\in C_0^\infty(\Omega_T)$ and $\varphi$ as in Definition \ref{def.renorm} in \eqref{4.aux3} and pass to the limit $\eps\to 0$. Proceeding as in \cite[Sec.~3.3]{LaVa18}, in particular using \eqref{4.gh} and integration by parts, the first two terms in \eqref{4.aux3} become after a computation
\begin{align*}
  \lim_{\eps\to 0}&\int_0^T\int_\Omega\Big\{
  \rho\widetilde{v}_m\cdot\pa_t
  \big(\overline{\phi\varphi'(\overline{\widetilde{v}_m}^\eps)}^\eps\big)
  + \rho (v\otimes\widetilde{v}_m):\na
  \big(\overline{\phi\varphi'(\overline{\widetilde{v}_m}^\eps)}^\eps\big)
  \Big\}dxdt \\
  &= \int_0^T\int_\Omega\rho\big(\varphi(\widetilde{v}_m)
  \pa_t\phi + \varphi(\widetilde{v}_m)v\cdot\na\phi\big)dxdt.
\end{align*}
We wish to pass to the limit $\eps\to 0$ in the fourth and fifth terms of \eqref{4.aux3}. To this end, we observe that
\begin{align*}
  \na\widetilde{v}_m &= \na\bigg(\frac{\phi_m(\rho)}{\rho}\rho v\bigg)
  = \na\bigg(\frac{\phi_m(\rho)}{\rho}\bigg)\rho v \\
  &\phantom{xx}+ \frac{\phi_m(\rho)}{\rho}\bigg(\sqrt{\frac{\rho}{\nu}}
  \mathbb{T}_\nu + 2\sqrt\rho v\otimes\na\sqrt\rho\bigg)
  \in L^2(\Omega_T;\R^{3\times 3}).
\end{align*}
This implies that $\na\varphi'(\widetilde{v}_m) = \varphi''(\widetilde{v}_m)\na\widetilde{v}_m\in L^2(\Omega_T;\R^{3\times 3})$ and the limit $\eps\to 0$ in the fourth term of \eqref{4.aux3} becomes
\begin{align*}
  -\lim_{\eps\to 0}&\int_0^T\int_\Omega
  \phi_m(\rho)\big(\sqrt{\nu\rho}\mathbb{S}_\nu
  + \sqrt{\kappa\rho}\mathbb{S}_\kappa\big):\na\big(
  \overline{\phi\varphi'(\overline{\widetilde{v}_m}^\eps)}^\eps\big)
  dxdt \\
  &= -\int_0^T\int_\Omega
  \phi_m(\rho)\big(\sqrt{\nu\rho}\mathbb{S}_\nu
  + \sqrt{\kappa\rho}\mathbb{S}_\kappa\big):\big(
  \na\phi\otimes \varphi'(\widetilde{v}_m) 
  + \phi\na\varphi'(\widetilde{v}_m)\big)dxdt.
\end{align*}
The limit $\eps\to 0$ in the remaining terms of \eqref{4.aux3} can be performed as well, and we end up with
\begin{align}\label{4.aux4}
  0 &= \int_0^T\int_\Omega\bigg(\rho\varphi(\widetilde{v}_m)\pa_t\phi
  + \rho\varphi(\widetilde{v}_m)v\cdot\na\phi  
  - \phi_m(\rho)\na p(\rho)\cdot\varphi'(\widetilde{v}_m)\phi \\
  &\phantom{xx}- \phi_m(\rho)\varphi''(\widetilde{v}_m)
  (\sqrt{\nu\rho}\mathbb{S}_\nu
  + \sqrt{\kappa\rho}\mathbb{S}_\kappa):\na\widetilde{v}_m \phi
  \nonumber \\
  &\phantom{xx}- (\sqrt{\nu\rho}\mathbb{S}_\nu
  + \sqrt{\kappa\rho}\mathbb{S}_\kappa):(\varphi'(\widetilde{v}_m)
  \otimes\na\phi_m(\rho))\phi \nonumber \\
  &\phantom{xx}-  (\sqrt{\nu\rho}\mathbb{S}_\nu
  + \sqrt{\kappa\rho}\mathbb{S}_\kappa):(\varphi'(\widetilde{v}_m)
  \otimes\na\phi) 
  + \rho\phi_m(\rho)\na c\cdot\varphi'(\widetilde{v}_m)
  - \frac{\rho\widetilde{v}_m}{\zeta}
  \cdot\varphi'(\widetilde{v}_m)\phi \nonumber \\
  &\phantom{xx}- \sqrt{\frac{\rho}{\nu}}\operatorname{tr}\mathbb{T}_\nu
  \rho\phi'_m(\rho) v\cdot\varphi'(\widetilde{v}_m)\phi
  - r_0\widetilde{v}_m\cdot\varphi'(\widetilde{v}_m)\phi
  - r_1\rho|v|^2\widetilde{v}_m\cdot\varphi'(\widetilde{v}_m)\phi
  \bigg)dxdt. \nonumber 
\end{align}

The next step is the limit $m\to\infty$. We deduce from $\phi_m(y)\to 1$, $\phi'_m(y)\to 0$ for a.e.\ $y>0$ and from $\rho>0$ a.e.\ in $\Omega_T$ (which is a consequence of $r_0\log\rho\in L^\infty(\Omega_T)$) that
\begin{align*}
  \phi_m(\rho)\to 1, \quad \rho^{3/4}\phi'_m(\rho)\to 0
  \quad\mbox{a.e. in }\Omega_T.
\end{align*}
We infer that $\widetilde{v}_m=\phi_m(\rho)v\to v$ a.e.\ and $g(\widetilde{v}_m)\to g(v)$ a.e.\ for any $g\in W^{1,\infty}(\R^3)$. We pass to the limit $m\to\infty$ in \eqref{4.aux4}, arguing by dominated convergence. This requires the pointwise a.e.\ convergence of the integrand, which is straightforward for all terms except for $\sqrt\rho\na\widetilde{v}_m$, appearing in the fourth term on the right-hand side of \eqref{4.aux4}. We reformulate this expression:
\begin{align*}
  \sqrt\rho\na\widetilde{v}_m
  &= \sqrt\rho\na(\phi_m(\rho)v)
  = \sqrt\rho\phi'_m(\rho)\na\rho\otimes v
  + \sqrt\rho\phi_m(\rho) \na v \\
  &= 4\rho^{3/4}\phi'_m(\rho)\na\rho^{1/4}\otimes(\sqrt\rho v)
  + \phi_m(\rho)\rho^{-1/2}(\na(\rho v)-\na\rho\otimes v) \\
  &= 4\rho^{3/4}\phi'_m(\rho)\na\rho^{1/4}\otimes(\sqrt\rho v)
  + \phi_m(\rho)\nu^{-1/2}\mathbb{T}_\nu. 
\end{align*}
Then, with the previous convergence results,
\begin{align*}
  \sqrt\rho\na\widetilde{v}_m\to \nu^{-1/2}\mathbb{T}_\nu
  \quad\mbox{a.e. in }\Omega_T.
\end{align*}
Because of $\phi_m(\rho)\le 1$ and $|\rho^{3/4}\phi'_m(\rho)|\le 2$, the integrand is dominated by an integrable function. Therefore, by dominated convergence, we conclude from \eqref{4.aux4} in the limit $m\to\infty$ that
\begin{align*}
  0 &= \int_0^T\int_\Omega\bigg(\rho\varphi(v)\pa_t\phi
  + \rho\varphi(v)v\cdot\na\phi - \na p(\rho)\varphi'(v)\phi \\
  &\phantom{xx}- (\sqrt{\nu\rho}\mathbb{S}_\nu 
  + \sqrt{\kappa\rho}\mathbb{S}_\kappa):(\varphi'(v)\otimes\na\phi)
  + \rho\na c\cdot\varphi'(v)\phi \\
  &\phantom{xx}- \frac{\rho v}{\zeta}\cdot\varphi'(v)\phi
   - r_0v\cdot\varphi'(v)\phi - r_1\rho|v|^2 v\cdot\varphi'(v)\phi
  \bigg)dxdt - \langle R_\varphi,\phi\rangle,
\end{align*}
where
\begin{align*}
  \langle R_\varphi,\phi\rangle
  = \frac{1}{\sqrt\nu}\int_0^T\int_\Omega
  \varphi''(v)\phi(\sqrt{\nu\rho}\mathbb{S}_\nu 
  + \sqrt{\kappa\rho}\mathbb{S}_\kappa):\mathbb{T}_\nu dxdt.
\end{align*}

It remains to verify formula \eqref{1.T} for $\mathbb{T}_\nu$. We multiply the $(j,k)$th component of \eqref{1.Tnu} by $\overline{\phi_m(\rho)\phi(\pa\varphi/\pa v_i)(v)}^\eps$ and pass to the limit $\eps\to 0$:
\begin{align*}
  \int_0^T&\int_\Omega\sqrt{\nu\rho}(\mathbb{T}_\nu)_{jk}
  \phi_m(\rho)\frac{\pa\varphi}{\pa v_i}(v)\phi dxdt \\
  &= \nu\int_0^T\int_\Omega\bigg(\frac{\pa}{\pa x_j}(\rho v_k)
  \phi_m(\rho)\frac{\pa\varphi}{\pa v_i}(v)\phi
  - 2\sqrt\rho v_k\frac{\pa\sqrt\rho}{\pa x_j}  
  \phi_m(\rho)\frac{\pa\varphi}{\pa v_i}(v)\phi\bigg)dxdt \\
  &= \nu\int_0^T\int_\Omega\bigg\{\frac{\pa}{\pa x_j}
  \bigg(\rho v_k\phi_m(\rho)\frac{\pa\varphi}{\pa v_i}(v)\bigg)\phi
  - \rho v_k\frac{\pa}{\pa x_j}\bigg(\phi_m(\rho)
  \frac{\pa\varphi}{\pa v_i}(v)\bigg)\phi \\
  &\phantom{xx}- 2\sqrt\rho v_k\frac{\pa\sqrt\rho}{\pa x_j}  
  \phi_m(\rho)\frac{\pa\varphi}{\pa v_i}(v)\phi\bigg\}dxdt \\
  &= -\nu\int_0^T\int_\Omega\bigg(\rho v_k\phi_m(v)
  \frac{\pa\varphi}{\pa v_i}(v)\frac{\pa\phi}{\pa x_j}
  - 4\rho\phi'_m(\rho)\sqrt{\rho}v_k\sqrt[4]{\rho}
  \frac{\pa\sqrt[4]{\rho}}{\pa x_j}\frac{\pa\varphi}{\pa v_i}(v)\phi \\
  &\phantom{xx}- \sum_{\ell=1}^3 \sqrt{\rho}v_k\phi_m(\rho)
  \frac{\pa^2\varphi}{\pa v_i\pa v_\ell}(v)\sqrt{\rho}
  \frac{\pa v_\ell}{\pa x_j}\phi
  - 2\sqrt\rho v_k\frac{\pa\sqrt\rho}{\pa x_j}\phi_m(\rho)
  \frac{\pa\varphi}{\pa v_i}(v)\phi\bigg)dxdt,
\end{align*}
where in the last step, we integrated by parts in the first term on the right-hand side and computed the partial derivative $\pa/\pa x_j$ in the second term. The expression $\sqrt\rho\pa v_\ell/\pa x_j$ is understood symbolically as $\nu^{-1/2}(\mathbb{T}_\nu)_{j\ell}$; see \cite[p.~205]{LaVa18}. Finally, since $\rho^{3/4}\phi'_m(\rho)\to 0$ a.e., the limit $m\to\infty$ leads to
\begin{align*}
  \int_0^T&\int_\Omega\sqrt{\nu\rho}(\mathbb{T}_\nu)_{jk}
  \phi_m(\rho)\frac{\pa\varphi}{\pa v_i}(v)\phi dxdt \\
  &= -\nu\int_0^T\int_\Omega\bigg(\rho v_k
  \frac{\pa\varphi}{\pa v_i}(v)\frac{\pa\phi}{\pa x_j} 
  - 2\sqrt\rho v_k\frac{\pa\sqrt\rho}{\pa x_j}
  \frac{\pa\varphi}{\pa v_i}(v)\phi\bigg)dxdt
  + \langle Q_\phi^{ijk},\phi\rangle,
\end{align*}
where
\begin{align*}
  \langle Q_\phi^{ijk},\phi\rangle
  = -\nu\sum_{\ell=1}^3\int_0^T\int_\Omega
  v_k\frac{\pa^2\varphi}{\pa v_i\pa v_\ell}(v)
  \sqrt{\frac{\rho}{\nu}}(\mathbb{T}_\nu)_{j\ell}\phi dxdt.
\end{align*}

\subsection{From renormalized to weak solutions}

To show that any renormalized solution is in fact a weak solution, we follow the approach in \cite[Sec.~3.2]{LaVa18}, since the chemotaxis term does not introduce any additional difficulties. For the convenience of the reader, we explain the main idea. Let $\Phi\in C^\infty(\R)$ be a nonnegative test function satisfying $\Phi(z)=1$ for $z\in[-1,1]$, $\Phi(z)=0$ for $z\in(-\infty,-2]\cup[2,\infty)$ and set $\Psi(z)=\int_0^z\Phi(s)ds$. Let $i\in\{1,2,3\}$ and define
\begin{align*}
  \varphi_n(y) = n\Psi(y_i/n)\prod_{j=1}^3\Phi(y_j/n)
  \quad\mbox{for }y\in\R^3.
\end{align*}
This function is an element of $W^{3,\infty}(\R^3)$ and fulfills condition \eqref{1.varphi} because it is compactly supported. We fix $i=1$. Then $\varphi(y)\to y_1$, $\varphi'(y)\to(1,0,0)$ a.e.\ as $n\to\infty$  and $\|\varphi''(y)\|_{L^\infty(\R^3)}\le C/n$. This bound implies that $\langle R_{\varphi_n},\phi\rangle\to 0$ and $\langle Q_{\varphi_n},\phi\rangle\to 0$. The mass and chemotaxis equations coincide in both (weak or renormalized) formulations. The limit in the momentum equation is obtained by an application of the dominated convergence theorem.

\subsection{Limit $(\kappa,r_0,r_1)\to 0$}

We consider a sequence of positive numbers $(\kappa_n,r_{0,n},r_{1,n})$ such that $(\kappa_n,r_{0,n},r_{1,n})\to 0$ as $n\to\infty$. Let $(\sqrt{\rho_n},\sqrt{\rho_n}v_n,c_n)$ be the renormalized solution to \eqref{1a.rhoc}--\eqref{1a.v} just constructed. The BD entropy inequality yields estimates that are independent of $n$ and, as shown in \cite[Lemma 5.1]{LaVa18}, up to subsequences, as $n\to\infty$,
\begin{align}
  \rho_n\to\rho &\quad\mbox{strongly in }
  L^\infty(0,T;L^q(\Omega))\mbox{ for }q<\max\{3,\gamma\}, 
  \nonumber \\
  \rho_nv_n\to \rho v &\quad\mbox{strongly in }
  L^2(0,T;L^r(\Omega))\mbox{ for }r<3/2, \nonumber \\
  \rho_n^\alpha H(v_n)\to \rho^\alpha H(v) &\quad\mbox{strongly in }
  L^s(\Omega_T)\mbox{ for }s<5\gamma/(3\alpha) \label{4.c3}
\end{align}
for any bounded continuous function $H$ and $0<\alpha<5\gamma/3$. Furthermore, our uniform bounds and the Aubin--Lions lemma imply that (again up to a subsequence)
\begin{align*}
  c_n\to c\quad\mbox{strongly in }L^\infty(0,T;L^2(\Omega)).
\end{align*} 
The limit $n\to\infty$ in the mass and chemotaxis equations is now straightforward. We examine the momentum equation term by term: 

\begin{itemize}
\item By \eqref{1.varphi}, the mappings $y\mapsto\varphi(y)$ and $y\mapsto y\varphi(y)$ are bounded and continuous. Hence, we deduce from \eqref{4.c3} that, for any $\gamma>1$,
\begin{align*}
  \rho_n\varphi(v_n)\to\rho\varphi(v), \quad
  \rho_n v_n\varphi(v_n)\to \rho v\varphi(v)
  \quad\mbox{strongly in }L^2(\Omega_T).
\end{align*}
\item A subsequence of $\na\rho_n^{\gamma/2}$ converges weakly to $\na\rho^{\gamma/2}$ in $L^2(\Omega_T)$ and $\rho_n^{\gamma/2}\varphi(v_n)$ converges strongly to $\rho^{\gamma/2}\varphi(v)$ in $L^s(\Omega_T)$ for $s<10/3$ (choosing $\alpha=\gamma/2$ in \eqref{4.c3}). This shows that
\begin{align*}
  \na\rho_n^\gamma\cdot\varphi'(v_n)
  = 2\na\rho_n^{\gamma/2}\cdot(\rho^{\gamma/2}\varphi'(v_n))
  \rightharpoonup \na\rho^\gamma\cdot\varphi'(v)
  \quad\mbox{weakly in }L^1(\Omega_T).
\end{align*}
\item The weak convergence of the sequences $(\sqrt{\rho_n}\mathbb{S}_\nu\varphi'(v_n))$ and $(\rho_n\na c_n\cdot\varphi'(v_n))$ follows immediately from the previous results.
\item The sequence $(\sqrt{\rho_n}\mathbb{S}_{\kappa,n})$ is bounded in $L^2(0,T;L^{3/2}(\Omega))$, implying that
\begin{align*}
  \sqrt{\kappa_n\rho_n}\mathbb{S}_{\kappa,n}\varphi'(v_n)
  \to 0\quad\mbox{strongly in }L^2(0,T;L^{3/2}(\Omega)).
\end{align*}
\item The sequence $(\sqrt{r_{0,n}}v_n\cdot\varphi'(v_n))$ is bounded in $L^2(\Omega_T)$, and $(r_{1,n}^{3/4}\rho_n|v_n|^2 v_n)$ is bounded in $L^{4/3}(\Omega_T)$. This shows that, for a subsequence,
\begin{align*}
  r_{0,n}v_n\cdot\varphi'(v_n)\to 0 &\quad
  \mbox{strongly in }L^2(\Omega_T), \\
  r_{1,n}\rho_n|v_n|^2 v_n\cdot\varphi'(v_n)\to 0 &\quad
  \mbox{strongly in }L^{4/3}(\Omega_T).
\end{align*}
\item We conclude from the bound
\begin{align*}
  \|R_{\varphi,n}\|_{\mathcal{M}(\Omega_T)}
  \le C\big(\|\mathbb{T}_{\nu,n}\|_{L^2(\Omega_T)},
  \|\mathbb{S}_{\nu,n}\|_{L^2(\Omega_T)},
  \|\mathbb{S}_{\kappa,n}\|_{L^2(\Omega_T)}\big)
  \|\varphi''\|_{L^\infty(\R^3)}
  \le C\|\varphi''\|_{L^\infty(\R^3)}
\end{align*}
that $R_{\varphi,n}\rightharpoonup R_\varphi$ weakly* in $\mathcal{M}(\Omega_T)$ for some measure $R_\varphi$. 
\end{itemize}

It remains to show identity \eqref{1.T} for $\mathbb{T}_\nu$. We know that $\mathbb{T}_\nu$ converges weakly in $L^2(\Omega_T)$ and that
\begin{align}\label{4.q}
  \sqrt{\rho_n}\frac{\pa\varphi}{\pa v_i}(v_n)\to
  \sqrt\rho\frac{\pa\varphi}{\pa v_i}(v)
  \quad\mbox{strongly in }L^q(\Omega_T)\mbox{ for }q<10\gamma/3.
\end{align}
Therefore, the left-hand side of \eqref{1.T} converges. In a similar way, the first term on the right-hand side converges thanks to \eqref{4.c3} (and using \eqref{1.varphi}). For the last term, we use the strong convergence \eqref{4.q} and the weak convergence $\na\sqrt{\rho_n}\rightharpoonup\na\sqrt\rho$ in $L^2(\Omega_T)$. Finally, the convergence of the measure $Q_\varphi$ follows by the same arguments as in the momentum equation. This completes the proof.


\end{document}